\documentclass[11pt]{amsart}
\usepackage{geometry}                
\usepackage{graphicx}
\usepackage{amssymb}
\usepackage{epstopdf}
\newtheorem{theorem}{Theorem}[section]
\newtheorem{lemma}{Lemma}[section]
\newtheorem{corollary}{Corollary}[section]
\newtheorem{proposition}{Proposition}[section]
\newtheorem{conjecture}{Conjecture}[section]
\newtheorem{definition}{Definition}[section]

\newtheorem{remark}{Remark}[section]

\numberwithin{equation}{section}

\def\Z{\Bbb Z}

\def\R{\Bbb R}
\def\P{\Bbb P}
\def\C{\Bbb C}

\def\H{\Bbb H}
\def\E{\Bbb E}
\def\S{\Bbb S}
\def\P{\Bbb P}

\def\D{\Delta}
\def\d{\partial}

\def\e{\epsilon}
\def\a{\alpha}
\def\b{\beta}
\def\g{\gamma}
\def\om{\omega}

\title{Holography of geodesic flows, harmonizing metrics, and billiards' dynamics}
\author{Gabriel Katz}
\address{5 Bridle Path Circle, Framingham, MA 01701, U.S.A.}
\email{gabkatz@gmail.com}

\begin{document}

\maketitle

\begin{abstract} 
Let $(M, g)$ be a Riemannian manifold with boundary, where $g$ is a non-trapping metric. Let $SM$ be the space of the spherical tangent to $M$ bundle, and $v^g$ the geodesic vector field on $SM$. We study the scattering maps $C_{v^g}: \partial^+_1SM \to \partial^-_1SM$, generated by the $v^g$-flow, and the dynamics of the billiard maps $B_{v^g, \tau}: \partial^+_1SM \to \partial^+_1SM$, where $\tau$ denotes an involution, mimicking the elastic reflection from the the boundary $\partial M$. We getting a variety of holography theorems that tackle the inverse scattering problems for $C_{v^g}$ and theorems that describe the dynamics of $B_{v^g, \tau}$. Our main tools are a Lyapunov function $F: SM \to \mathbb R$ for $v^g$ and a special harmonizing Riemannian metrics $g^\bullet$ on $SM$, a metric in which $dF$ is harmonic. For such metrics $g^\bullet$, we get a family of isoperimetric inequalities of the type $vol_{g^\bullet}(SM) \leq vol_{g^\bullet |}(\d(SM))$ and formulas for the average volume of the minimal hypesufaces $\{F^{-1}(c)\}_{c \in F(SM)}$. We investigate the interplay between the harmonizing metrics $g^\bullet$ and the classical Sasaki metric $gg$ on $SM$. Assuming ergodicity of $B_{v^g, \tau}$, we also get Santal\'{o}-Chernov type formulas for the average length of free geodesic segments in $M$ and for the average variation of the Lyapunov function $F$ along the $v^g$-trajectories.  
\end{abstract}

\section{Introduction}

This paper is an extension of \cite{K5}, where we proposed  ``a more topological  approach" to some classical inverse scattering problems. Here we take a similar view of the geometry of scattering maps and of the dynamics of billiard maps. To validate our approach, we need to built some infrastructure in the land of traversing flows  on general manifolds with boundary \cite{K1}, \cite{K2}, \cite{K7}. Then    
we employ this infrastructure to study the geodesic flows.
This effort, in the spirit of ``Integral Geometry and Geometric Probability" by Santal\'{o} (\cite{S}, Chapter 19), and works by Vidal Abascal \cite{V}-\cite{V2}, is in the core of the present paper.  \smallskip

Before we describe our results, let us start with a brief review of different classes of vector fields on manifolds with boundary that occur in the paper (as a general reference, see \cite{K7}). Let $X$ be a connected compact smooth $(n+1)$-dimensional manifold with boundary. A smooth vector field $v$ on $X$ is called {\sf traversing} if it admits a {\sf Lyapunov function} $F: X \to \R$ such that $dF(v) > 0$ everywhere in $X$. 

The trajectories of a traversing vector field are homeomorphic to closed segments or to singletons. Conversely, by \cite{K1}, Corollary 4.1, if all the $v$-trajectories are homeomorphic to closed segments or to singletons, then the field admits a Lyapunov function, and thus is a traversing vector field of the gradient type. 
  
For a smooth vector field $v$ on a smooth compact $(n+1)$-dimensional manifold $X$, such that $v \neq 0$ along $\d X$, let us introduce an important {\sf Morse stratification} $\{\d_j^\pm X(v)\}_{j \in [1, \dim X]}$ of the boundary $\d X$ \cite{Mo}. The stratum $\d_jX =_{\mathsf{def}} \d_jX(v)$ has the following description (see \cite{K1}) in terms of an auxiliary function $z: \hat X \to \R$ that satisfies the three properties:
\begin{eqnarray}\label{eq2.3}
\end{eqnarray}

\begin{itemize}
\item $0$ is a regular value of $z$,   
\item $z^{-1}(0) = \d X$, and 
\item $z^{-1}((-\infty, 0]) = X$. 
\end{itemize}

Employing  $z$, the locus $\d_jX =_{\mathsf{def}} \d_jX(v)$ is defined by the equations: 
$$\big\{z =0,\; \mathcal L_vz = 0,\; \ldots, \;  \mathcal L_v^{(j-1)}z = 0\big\},$$
where $\mathcal L_v^{(k)}$ stands for the $k$-th iteration of the Lie derivative operator $\mathcal L_v$ in the direction of $v$. 
The pure stratum $\d_jX^\circ \subset \d_jX$ is defined by the additional constraint  $\mathcal L_v^{(j)}z \neq 0$. The locus $\d_jX$ is the union of two loci: {\bf (1)} $\d_j^+X$, defined by the constraint  $\mathcal L_v^{(j)}z \geq  0$, and {\bf (2)} $\d_j^-X$, defined by the constraint  $\mathcal L_v^{(j)}z \leq  0$.  

\begin{definition}\label{def.boundary_generic_v}
The vector field  $v$ is {\sf boundary generic} with respect to $\d X$ if the $j$-form
\begin{eqnarray} \label{eq.generic_simple}
dz \wedge d(\mathcal L_vz) \wedge \; \ldots \;  \wedge\; d(\mathcal L_v^{(j-1)}z)
\end{eqnarray}
represents a nonzero section of the bundle $\bigwedge^{j}T_\ast X$ along the locus  ${\d_jX}$ \;  for all $j \in [1, n+1]$.
\end{definition}
For a boundary generic $v$, all the strata $\d_j X$ are smooth $(n+1-j)$-manifolds and the two loci, $\d_j^+X$ and $\d_j^-X$, share a common boundary $\d_{j+1}X$.  If $v$ on $X$ is boundary generic, then each point $x \in \d X$ belongs to a unique minimal stratum $\d_j X \subset \d X$ with a maximal $j = j(x) \leq n+1$.  



In \cite{K2}, Definition 3.2, we introduced another class of vector fields on $X$ which we call {\sf traversally generic}. They are a subclass of the traversing and boundary generic vector fields. Loosely speaking, for a traversally generic $v$, the localized projection of $\d X$ on a transversal to the $v$-flow section $S$ is a Thom-Boardman map (\cite{Bo}) of the combinatorial type $(1, 1, \dots , 1)$. By Theorem 3.5 from \cite{K2}, the traversally generic vector fields form also an open and dense set in the space of all traversing fields. 

Any trajectory $\g$ of a boundary generic and traversing vector field $v$ generates a finite sequence $\omega = (\omega_1, \dots , \omega_q)$
of natural numbers,  the entries of the sequence correspond to $v$-ordered points of the finite set $\g \cap \d X$.  Each point $x \in \g \cap \d X$ contributes to $\omega$ a natural number $j(x)$, the \textsf{multiplicity of tangency} of $\g$ to the boundary $\d X$ at $x$. In fact, $j(x)$ is the index $j$ of the smallest stratum $\d_jX(v)$ to which $x$ belongs. The ordered list $\omega_\g$ of these multiplicities is \textsf{the combinatorial type} of $\g$. For boundary generic and traversing vector fields, the combinatorial types $\omega_\g$ of their trajectories belong to an universal ($X$-independent) poset $\mathbf \Omega^\bullet$, while for traversally generic vector fields, the combinatorial types $\omega_\g$ belong to a subposet $\mathbf \Omega^\bullet_{'\langle n]}$ (see \cite{K4} for its definition and properties). Remarkably, for the traversally generic vector fields, the combinatorial type $\omega_\g$ determines the smooth topological type of the $v$-flow in the vicinity of $\g \subset X$ (\cite{K2}). 

Let $\g_x$ denote the $v$-trajectory through $x \in X$. Any traversing vector field $v$ on $X$ produces a, so called, \textsf{causality map} $C_v$ which takes a portion $\d_1^+X(v)$ of the boundary $\d X$ to the closure $\d_1^-X(v)$ of the complementary portion (see Fig.1). Here $\d_1^\pm X(v)$ stands for the locus in $\d_1 X := \d X$, where $v$ is directed inward/outward of $X$ or is tangent  to $\d X$. By definition, $C_v(x)$ is the point $y \in \d_1^-X(v)$ that resides in $\g_x \cap \d X$ above the point $x \in \d_1^+X(v)$. When no such $y$ exists, we put $C_v(x) = x$.  We stress that, in general, $C_v$ is a \emph{discontinuous map}. 

For the reader convenience, we state Theorem 3.1 from \cite{K4}, crucial for our efforts here.
\begin{theorem}\label{th1.1}{\bf (The Holography Theorem).}
Let $X_1, X_2$ be two smooth compact connected $(n +1)$-manifolds with boundary, equipped with traversing boundary generic vector fields  $v_1, v_2$, respectively. 
\begin{itemize}
\item Then any smooth diffeomorphism  $\Phi^\d: \d X_1 \to \d X_2$, such that $$\Phi^\d \circ C_{v_1} =  C_{v_2} \circ \Phi^\d,$$ extends to a homeomorphism $\Phi: X_1 \to X_2$ which maps $v_1$-trajectories to $v_2$-trajectories so that the field-induced orientations of trajectories are preserved. The restriction of $\Phi$ to each trajectory is a smooth diffeomorphism. \smallskip

\item If each $v_2$-trajectory is either transversal to $\d X_2$ at \emph{some} point, or is simply tangent to $\d X_2$,\footnote{This is the case when the $v$-flow has no trajectories of the combinatorial types $\omega \in (4)_\succeq \bigcup\, (33)_\succeq$} then the homeomorphism $\Phi$ is a smooth diffeomorphism. In particular,  $\Phi$ is a smooth diffeomorphism when $\d X_2$ is \emph{concave} with respect to the $v_2$-flow. \hfill $\diamondsuit$
\end{itemize}
\end{theorem}

\noindent {\bf Remark 1.1.} The hypothesis in the second bullet of Theorem \ref{th1.1} are perhaps superfluous: we conjecture that the conjugating homeomorphism $\Phi: X_1 \to X_2$ is always a diffeomorphism. \hfill $\diamondsuit$ 

\begin{definition}\label{def1.1}  A Riemannian metric $g$ on a connected compact manifold $M$ with boundary is called \textsf{non-trapping} if $(M, g)$ has no closed geodesics and no geodesics of an infinite length (the later are homeomorphic to an open or a semi-open interval). \hfill $\diamondsuit$ 
\end{definition}

Let $M$ be a compact connected smooth Riemannian manifold with boundary. We assume that the metric $g$ on $M$ is {\sf non-trapping}. In fact, $g$ is non-trapping if and only if the geodesic flow $v^g$ on the spherical tangent bundle $SM = SM(g)$ admits a Lyapunov function $F: SM \to \R$ so that $dF(v^g) > 0$ \cite{K5}. 
The space $\mathcal G(M)$ of non-trapping Riemannian metrics on $M$ forms an open set in the space of all Riemannian metrics.  

In \cite{K5}, we introduce a class of Riemannian metrics $g$ on $M$ which we call  \textsf{geodesically boundary generic} or {\sf boundary generic} for short. 

\begin{definition}\label{def.boundary_generic_g}
A metric $g$ on $M$ is {\sf boundary generic}, if the geodesic vector field $v^g$ on $SM$ is boundary generic  with respect to $\d(SM)$ in the sense of Definition \ref{def.boundary_generic_v}.
\hfill $\diamondsuit$ 
\end{definition}

For a boundary generic metric $g$, the boundary $\d M$ is ``generically curved" in $g$. In particular, if each component of $\d M$ is strictly convex or concave in $g$, then $g$ is boundary generic.  The metrics $g$ in which $\d M$ is geodesically closed represent the extreme failure to be boundary generic. 

We speculate that the space $\mathcal G^\dagger(M)$ of geodesically boundary generic non-trapping metrics is open and dense in the space of all non-trapping metrics $\mathcal G(M)$ and prove that it is indeed open (\cite{K5}). \smallskip

Here is one of the main results from \cite{K5}, which we will use on many occasions (see Fig.2).

\begin{theorem}\label{th1.2}{\bf (The topological rigidity of the geodesic flow for the inverse scattering problem).}\smallskip

Let $(M_1, g_1)$ and  $(M_2, g_2)$ be two smooth compact connected Riemannian $n$-manifolds with boundaries.  Let the metrics $g_1$, $g_2$ be geodesically boundary generic, and let $g_2$ be non-trapping.
 
Assume that the scattering maps
$$C_{v^{g_1}}: \d_1^+(SM_1)(v^{g_1}) \to \d_1^-(SM_1)(v^{g_1})\;\; \text{and} \;\; C_{v^{g_2}}: \d_1^+(SM_2)(v^{g_2}) \to \d_1^-(SM_2)(v^{g_2})$$ are conjugate by a smooth diffeomorphism $\Phi^\d: \d_1(SM_1) \to \d_1(SM_2)$.  \smallskip

Then $g_1$ is also non-trapping, and $\Phi^\d$ extends to a homeomorphism $\Phi: SM_1 \to SM_2$, which takes each $v^{g_1}$-trajectory to a $v^{g_2}$-trajectory. Moreover, $\Phi$, being restricted to any $v^{g_1}$-trajectory, is an orientation-preserving smooth diffeomorphism.
\smallskip

If  the metric $g_2$ is such that any geodesic curve in $M_2$ is either transversal to $\d M_2$ at \emph{some} point or is simply tangent to $\d M_2$, then $g_1$ must have the same property, and the conjugating homeomorphism $\Phi: SM_1 \to SM_2$ is a diffeomorphism.    \hfill $\diamondsuit$
\end{theorem}

One of the goals of this paper is to study close relatives of the scattering maps $C_{v^g}$, the \textsf{billiard maps} $B_{v^g}$ and their dynamics. The map $B_{v^g}$ is obtained from $C_{v^g}$ by composing it with the reflection diffeomorphism $\tau_g: \d(SM) \to \d(SM)$ that takes any unitary vector, tangent to $M$ at a point from $\d M$, to its mirror image, the boundary $\d M$ being the mirror.\bigskip

\begin{itemize}
\item {\it Now, let us describe some of our results in the order they appear in the paper.}  
\end{itemize}

{\sf In Section 2}, we show how special closed differential $n$-forms $\Theta$ on a compact $(n+1)$-dimensional manifold $X$ generate a measure $\mu_\Theta$ on $\d X$ such that the  causality map $C_v: \d_1^+X(v) \to \d_1^-X(v)$ is a measure-preserving transformation (Theorem \ref{prop.C_v_preserves}). \smallskip

{\sf In Section 3}, we combine the causality maps $C_v$ with measure-preserving involutions $\tau$ on the boundary $\d X$ to introduce {\sf proto-billiard maps} $B_{v, \tau}:  \d_1^+X(v) \to \d_1^+X(v)$,  dynamical measure-preserving systems (Theorem \ref{proto-billiard}).  \smallskip

{\sf In Section 4}, we deal with {\sf intrinsically harmonic} Lyapunov functions $f: X \to \R$ and Lyapunov $1$-forms $\a$ on $X$, specially adjusted to the given vector field $v$  (Theorem \ref{th2.1}). They go hand in hand with so called $v$-{\sf harmonizing} metrics $g$ on $X$ (see Definition \ref{harmonizing pair}). The $v$-harmonizing pairs $(g, \a)$ or $(g, df)$ each produces a pair of mutually orthogonal {\sf minimal (taut) foliations} $\mathcal F(v)$, $\mathcal G(\a)$ of dimensions $1$ and $n$, respectively (Corollary \ref{foliations}). 

Theorem \ref{traversing_harmonizing} claims that, for a traversing boundary generic $v$, there exists a $v$-harmonizing pair $(g, df)$, such that the $n$-form $\Theta = \ast_g(df)$ defines a measure on $\d X$ with respect to which $C_v$ is a measure-preserving map. Here ``$\ast_g$" denotes the {\sf Hodge star operator}. 

For a traversing vector field $v$, the differential form $\Theta$ helps also to define, in the spirit of \cite{S}, a measure $\mu_\Theta$ on {\sf the space of trajectories} $\mathcal T(v)$ (see Definition \ref{def2.4} and \cite{K3}). For a traversing vector field $v$ and a $v$-harmonizing pair $(g, df)$, Corollary \ref{iso_for_X} establishes the inequality 
 $$vol_\Theta(\mathcal T(v))\; \leq \; \frac{1}{2} vol_{g|}(\d X).$$
  If, in addition, we normalize the Lyapunov function $f$ so that it takes values in the interval $[0,1]$, then we also get an {\sf isoperimetric inequality} $$vol_g(X)\; \leq \; vol_{g|} (\d X),$$ valid for any $v$-harmonizing $g$. 
 
Assuming that the $v$-harmonizing metric $g$ and differential $df$ are $v$-invariant,  Theorem \ref{th2.2} describes the residual structures on the boundary $\d X$ that allow for a reconstruction of $X, v$ and $g$, up to a diffeomorphism of $X$. This theorem is the first among several results that we call {\sf holographic} (see also \cite{K5}, \cite{K8}).  \smallskip

{\sf In Section 5}, we apply the results about general traversing vector flows from the previous sections to the geodesic flows $v^g$ on the space of tangent spherical bundle $SM \to M$, where $M$ is a compact manifold with boundary. The $v^g$-flow is generated by a non-trapping boundary generic metric $g$ on $M$. 

Corollary \ref{hyperbolic} provides numerous examples of non-trapping metrics on codimension zero compact submanifolds of hyperbolic and Euclidean spaces.

In Theorem \ref{foliations on SM}, for a boundary generic non-trapping metric $g$ on $M$, 
we construct a $v^g$-harmonizing and $v^g$-invariant metric $g^\bullet$ on $SM$ and a {\sf well-balanced} (see Definition \ref{well-balanced}) Lyapunov function $F: SM \to \R$, so that $1$-dimensional foliation $\mathcal F(v^g)$ and orthogonal to it $(2n-2)$-dimensional foliation $\mathcal G(F) := \{F^{-1}(c)\}_{c \in \R}$ are {\sf minimal} in $g^\bullet$. Moreover, we prove that the scattering map $C_{v^g}: \d_1^+(SM) \to \d_1^-(SM)$ preserves the measure, defined by the harmonic form $\Theta: = \ast_{g^\bullet}(dF)$.

Let $g$ be as  above.  In Theorem \ref{th3.2}, for a given $v^g$-invariant volume form $\Omega$ on $SM$ and a Lyapunov function $F$, we use the form $\Theta:= v^g\, \rfloor \, \Omega$ to construct a $v^g$-harmonizing metric $g^\bullet$ such that $\ast_{g^\bullet}(dF) = \Theta$.  Let $vol_{\Theta}(\mathcal T(v^g))$ denote the $\Theta$-induced volume of {\sf the space of geodesics} $\mathcal T(v^g)$. Then we prove the inequality $$vol_{\Theta}(\mathcal T(v^g))\; \leq \; \frac{1}{2} vol_{g^\bullet |}(\d(SM)).$$  Normalizing $F$ so that $F(SM) \subset [0, 1]$, we get the isoperimetric inequality $$vol_{g^\bullet}(SM)\; \leq \; vol_{g^\bullet |}(\d(SM)).$$  

In holography Theorem \ref{th3.3}, we reconstruct the space $SM$, the vector field $v^g$, and the $v^g$-invariant harmonizing metric $g^\bullet$, up to a diffeomorphism of $SM$, in terms of some enhanced scattering data. \smallskip

{\sf In Section 6}, we study generalized billiard maps $B_{v^g, \tau}: \d_1^+(SM) \to \d_1^+(SM)$ on billiard tables $(M, g)$, were $g$ boundary generic and non-trapping (see Theorem \ref{main_theorem} and Corollary \ref{infinite_return_for_billiard}). We employ the restrictions to $SM$ of the fundamental Liouville $1$-form $\b_g$ and the symplectic $2$-form $\om_g = d\b_g$.  In holography Theorem \ref{eq.beta}, we show how to reconstruct the form $\b_g$ from its restriction $\b_g^{\d^+}$ to the boundary $\d(SM)$, the restriction $F|_{\d(SM)}$ of the Lyapunov function $F: SM \to \R$, and the scattering map $C_{v^g}$. \smallskip

{\sf In Section 7}, we study the dynamics of {\sf ergodic} (see Definition \ref{def9.7}) billiard maps $B_{v^g, \tau}$ on the billiards $(M, g)$, were $g$ is boundary generic and non-trapping. We apply the classical Birkhoff Theorem \cite{Bi} to the measure $\mu_\Theta$ on $\d(SM)$, generated by an appropriately constructed closed and $v^g$-invariant $(2n-2)$-form $\Theta$ (see Theorem \ref{ergodic_billiard}). 

In Theorem \ref{F-variation}, we compute the {\sf spacial and time averages} of the variation $\D_F$ of the Lyapunov function $F: SM \to \R$ along the $v^g$-trajectories; for the ergodic billiards, both computations produce the same result.  
Theorem \ref{balanced F-variation} is a version of these computations for a given Lyapunov function $F$ and the form $\Theta = \om_g^{n-1}$, the $(n-1)^{st}$ exterior power of the symplectic $2$-form $\om_g$ on $SM$. It employs a $v^g$-harmonizing metric $g^\bullet$ on the space $SM$. In contrast, Theorem \ref{ell(gg)} is a result of a similar calculation of the {\sf average length} of the $v^g$-trajectories in the {\sf Sasaki metric} $gg$ on $SM$ (\cite{Sa}). It utilizes the same form $\Theta = \om_g^{n-1}$. In this case, the results of the computation can be expressed directly in terms of the volumes $vol_g(M)$ and $vol_{g|_{\d M}}(\d M)$ (see formulas (\ref{ell_gg}) and (\ref{B})). For any non-trapping $g$, this leads to the following inequality (Corollary \ref{diameter}): $$vol_g(M)\; \leq \; c(n) \cdot \mathsf{gd}(M, g) \cdot vol_{g|}(\d M),$$  where $c(n) > 0$ is a well-known universal constant, and $\mathsf{gd}(M, g)$ denotes the maximal length of {\sf free} geodesic arcs in $M$ (see Fig.3). The later inequality resembles one classical inequality from \cite{Cr}, where $\mathsf{gd}(M, g)$ is replaced by the diameter of $M$. 

In Theorem \ref{A(g_bullet)}, we derive formulas for computing the volume $A_{g^\bullet}(c)$ of a taut slice $F^{-1}(c)$ ($c \in \R$) in the $v^g$-harmonizing metric $g^\bullet$, as well as the average value of $A_{g^\bullet}(c)$. 

 
\smallskip
{\it Acknowledgments:} I am grateful  to Christofer Croke, Serge Tabachnikov, and Gunther Uhlmann for very enlightening conversations.


\section{Causality maps of traversing flows as measure-preserving transformations}






Let $X$ be a connected compact smooth $(n+1)$-manifold with boundary, and $v$ a smooth traversing and boundary generic vector field on $X$. 
\smallskip
\begin{figure}[ht]\label{fig9.3}
\centerline{\includegraphics[height=1.7in,width=2.4in]{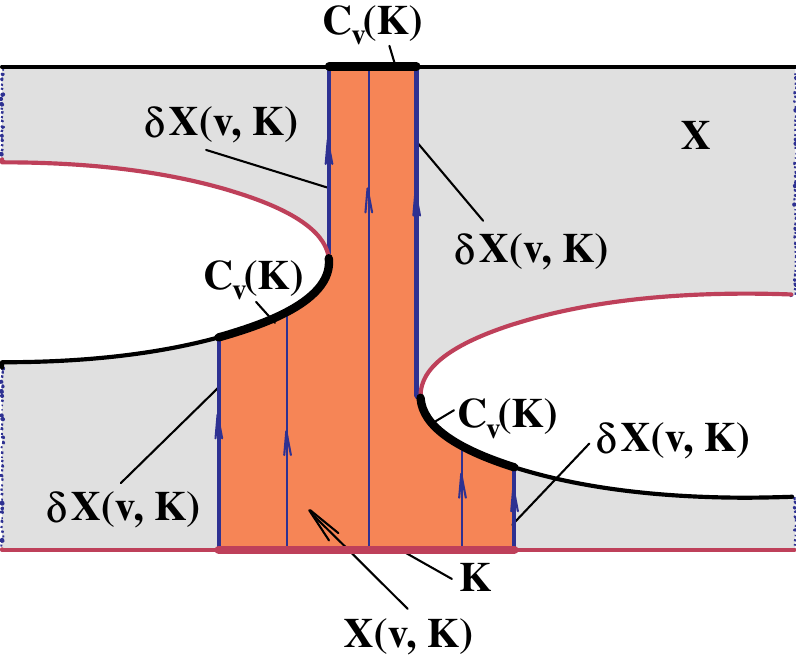}}
\bigskip
\caption{\small{The set $X(v, K)$ and its boundary for a codimension zero submanifold  $K \subset \d_1^+X$. Note the image $C_v(K)$ of $K$ under the map $C_v$.}} 
\end{figure}

As the lemma below testifies, the causality maps $$C_v: \d_1^+X(v) \to \d_1^-X(v),$$ although discontinuous, have  some ``positive features": they preserve certain $n$-dimensional measures on the $n$-manifold $\d X$, the measures that are amenable to $v$.  

\begin{lemma}\label{lem2.1} Let $X$ be a compact smooth oriented $(n+1)$-manifold with boundary, carrying a boundary generic traversing vector field $v$. 
We denote by $\Omega^\d$ a positive volume $n$-form on $\d X$, consistent with its orientation. Let $\Theta \in \bigwedge^{n}(T^\ast X)$ be a differential $n$-form on $X$, subject to the constraints:
\begin{enumerate}
\item $\Theta(v \wedge w)= 0$ for any  polyvector $w \in \bigwedge^{n-1}(TX)$,
\item $d \Theta = 0$,
\item the function $(\Theta \big/ \Omega^\d): \d X \to \R$ is nonnegative on $\d_1^+X(v)$ and nonpositive on $\d_1^-X(v)$.
\end{enumerate}
We denote by $\mu^\d$ the Lebesgue measure on $\d X$, induced by some Riemannian metric on $\d X$. Then the following properties are valid:
\begin{itemize}
\item  the Lie derivative $\mathcal L_v(\Theta)  = 0$,
\item restricting the $n$-form $\Theta$ to the boundary $\d X$, for any $\mu^\d$-measurable set $K \subset \d_1^+X(v)$, we get
\begin{eqnarray}\label{eq2.1}
\int_K \Theta  = \Big|\int_{C_v(K)} \Theta \Big|.
\end{eqnarray}
\end{itemize}
\end{lemma}

\begin{proof} Using the identity $\mathcal L_v(\Theta) =   d(v \rfloor \Theta) + v \rfloor d\Theta$ for the Lie derivative and properties $(1)$ and $(2)$ of $\Theta$ from the lemma hypotheses, we conclude that $\mathcal L_v(\Theta)  = 0$, that is, the form $\Theta$ is invariant under the $v$-flow. Thanks to property $(1)$, $\Theta$ is a ``horizontal" form.
\smallskip

Let $v$ be a boundary generic traversing vector field and $\Theta$ an $n$-form on $X$ as in the hypotheses of the lemma.  For any Lebesgue-measurable $K \subset \d X$, we \emph{define} its measure $\mu_\Theta(K)$  by the formula
\begin{eqnarray}\label{eq2.2} 
\mu_\Theta(K) =_{\mathsf{def}} \int_{K \cap \d_1^+X(v)} \Theta\; -  \int_{K \cap \d_1^-X(v)} \Theta.
\end{eqnarray}
Note that this formula makes sense since, for a boundary generic $v$, the sets $\d^\pm_1X(v)$ are smooth manifolds, and the intersection of two Lebesgue-measurable sets is again Lebesgue-measurable.\smallskip

For any set $A \subset \d_1^+X(v)$, we denote by  $X(v, A)$ the set, formed by the $v$-trajectories through the points of $A$. 

Consider the locus $\d_2^+X(v) \subset \d(\d_1^+X(v))$, the closure of points of the boundary $\d X$, where the $v$-flow is simply tangent to the boundary and the boundary is {\sf concave} with respect to the flow in the sense of \cite{K1}. For a boundary generic field $v$, let $\mathcal X$ denote the set $X(v, \d_2^+X(v))$, the union of $v$-trajectories that contain points from the locus $\d_2^+X(v)$.  Then $\mathcal X \cap \textup{int}(\d_1^+X(v))$ is the discontinuity locus of the causality pap $C_v$.

For a boundary generic $v$, Lemmas 3.1 and 3.4 from \cite{K2} provide us with local real semi-algebraic models of the domain and range of the causality map $C_v$, as well as with local real analytic models of the causality map itself away from the $(n-1)$-dimensional $\mathcal X \cap \textup{int}(\d_1^+X(v))$. The word ``local" here means ``in the vicinity of each $v$-trajectory". These local models imply, in particular, that $\mu^\d(C_v(A)) = 0$ for any set $A \subset \d_1^+X$ whose Lebesgue measure $\mu^\d(A) = 0$. 
They imply also that 
the $C_v$-image of a Lebesgue measurable set $K \subset \d_1^+X$ is Lebesgue measurable in $\d_1^-X$. 
\smallskip

In order to prove that the sum of the Lebesgue integrals $$\int_K \Theta + \int_{C_v(K)} \Theta \; = \; 0$$ for any Lebesgue measurable $K \subset \d_1^+X(v)$, it will suffice to show that $\int_N \Theta + \int_{C_v(N)} \Theta = 0$ for all $n$-dimensional compact piecewise differentiable (``$\mathsf{PD}$" for short) manifolds $N \subset \d_1^+X(v)$.


For any $\mathsf{PD}$-submanifold $N \subset \d_1^+X$, we form the set $X(v, N)$. Again, thanks to the local models of Lemmas 3.1 and 3.4 from \cite{K2}, the locus $X(v, N)$ is a piecewise differentiable manifold with boundary. 
Its oriented boundary $\d X(v, N)$ is formed by the three $\mathsf{PD}$-manifolds: $N$, $C_v(N) \subset \d_1X$, and the rest, which we denote $\delta X(v, N)$. The latter is built out of segments of $v$-trajectories (see Fig.1). 
So, by the Stokes Theorem, $$\int_{X(v, N)} d\Theta \; = \; \int_{N} \Theta + \int_{C_v(N)}\Theta + \int_{\delta X(v, N)}\Theta.$$
Since $d\Theta = 0$, we get $\int_{X(v, N)} d\Theta = 0$. Using that $\Theta$ is vertical and $\delta X(v, N)$ consists of $v$-trajectories, we get $$\int_{\delta X(v, N)}\Theta = 0,$$  which implies formula (\ref{eq2.1}) and the measure-defining formula (\ref{eq2.2}). 
%
%
%
%
%
\hfill
\end{proof}

\begin{definition}\label{def.dual} Let $v$ be a smooth non-vanishing vector field on a compact orientable $(n+1)$-dimensional manifold $X$ and $\mathcal H$ a $n$-dimensional distribution, transversal to $v$. 
 Consider a 
 differential $n$-form $\Theta$ on $X$. 
\begin{enumerate}
\item We say that $\Theta$ \textsf{integrally dual} to $v$, if: 
\begin{itemize}
\item $d\Theta = 0$ 
\item the kernel of $\Theta$ is a $1$-dimensional distribution on $X$,
\item $v \in \ker(\Theta)$,
\item $\pm\Theta |_{\d_1^\pm X(v)} \geq 0$.
\end{itemize}
\smallskip

\item We call $\Theta$ \textsf{integrally dual} to the pair $(v, \mathcal H)$, if  $\Theta$ integrally dual to $v$ and $\Theta|_{\mathcal H} > 0$ with respect to the orientation of $\mathcal H$, induced by $v$ and the orientation of $X$.\footnote{It follows from $\Theta|_{\mathcal H} > 0$ that the kernel $\ker(\Theta)$ is a $1$-dimensional distribution.}
 \hfill $\diamondsuit$ 
\end{enumerate} 
\end{definition} 

With Definition \ref{def.dual} in place, we may rephrase  Lemma \ref{lem2.1} as follows: 

\begin{proposition}\label{prop.C_v_preserves}
For a traversing  vector field $v$ on a compact smooth $(n+1)$-manifold $X$, any integrally dual to $v$ $n$-form $\Theta$ is $v$-invariant. Moreover, $\Theta$ defines a measure $\mu_\Theta$ on $\d  X$ such that the causality map $C_v$ is this measure preserving transformation. \hfill $\diamondsuit$
\end{proposition}

The rest of this section contains few elementary lemmas to be used in what follows.

\begin{lemma}\label{unique} Let $X$ be a smooth compact oriented $(n+1)$-manifold. 
\begin{itemize}
\item For a given non-vanishing vector field $v$ on $X$, an integrally dual form $\Theta$ is unique, up to the multiplication by a smooth function $h: X \to \R$ such that $dh \wedge \Theta = 0$ in $X$, and $h|_{\d X} \geq 0$. 
\smallskip

\item For a given non-vanishing vector field $v$ and a transversal to it $n$-distribution $\mathcal H$ on $X$, a form $\Theta$, integrally dual to $(v, \mathcal H)$, is unique, up to the multiplication by a smooth positive function $h: X \to \R_+$ such that $dh \wedge \Theta = 0$ in $X$. 
\end{itemize}
\end{lemma}

\begin{proof} Consider the $n$-dimensional bundle $\bigwedge^n T^\ast X \to X$. 
The linear constraints on a form $\theta \in \bigwedge^n T^\ast X$, imposed by the property $\big\{v \in \ker(\theta)\big\}$,  define a $1$-dimensional subbundle $\Lambda_v$ of the bundle $\bigwedge^n T^\ast X$. Let us denote by $\Gamma(\Lambda_v)$ the linear space of smooth sections-forms  $\Theta$ of the line bundle $\Lambda_v \to X$. Let $\Gamma^\star(\Lambda_v)$ denote the subspace of nowhere vanishing sections from $\Gamma(\Lambda_v)$. Consider the kernel $\mathcal K_v$ of the differential $$d: \Gamma(\Lambda_v) \longrightarrow \Gamma(\bigwedge^{n+1} T^\ast X) \approx C^\infty(X, \R).$$ 
Since $\dim(\Lambda_v) =1$, any two sections $\Theta, \Theta' \in \Gamma(\Lambda_v)$ differ by a functional multiple, i.e., $\Theta' = h \cdot \Theta$. 

When $\Theta, \Theta' \in \mathcal K_v \cap \Gamma(\Lambda_v)$, using that $d\Theta = 0$, we get $0= d(h \Theta) = dh\wedge \Theta$. So the functional coefficient $h: X \to \R$ is such that $dh \wedge \Theta = 0$ identically in $X$. The requirement $\pm h \cdot\Theta |_{\d_1^\pm X(v)} \geq 0$ implies that $h: \d X \to \R$ must be nonnegative on $\d_1X$.

When $\Theta, \Theta' \in \mathcal K_v \cap \Gamma^\star(\Lambda_v)$ are integrally dual to $(v, \mathcal H)$, the both forms are nonsingular and $h > 0$ on $X$. 
Therefore, for a positive $h$, if $\pm\Theta |_{\d_1^\pm X(v)} \geq 0$, then $\pm\Theta' |_{\d_1^\pm X(v)} \geq 0$ as well. 
\end{proof}

\noindent {\bf Remark 2.1.}
We stress that Lemma \ref{unique} does not claim the \emph{existence} of an integrally dual form $\Theta$ for a given $v$ or $(v, \mathcal H)$. It follows from Theorem \ref{th2.1} that, for any non-vanishing vector field $v$, 
there is no \emph{local} obstruction to the existence of integrally dual form $\Theta$. In fact, at least for any {\sf invariant Calabi's vector field} $v$ (as in Definition \ref{Calabi_field}, second bullet) admits such a form. On the other hand, by Lemma \ref{nildivergent_form}, the existence of an integrally dual to $v$ form $\Theta$ is equivalent to an existence of a $v$-invariant volume form $\Omega$ on $X$.
\hfill $\diamondsuit$
\smallskip



\begin{lemma}\label{cone} For a given non-vanishing vector field $v$, the space  $\mathcal D(v)$ of $n$-forms $\Theta$, integrally dual to $v$, is convex in the space of all $n$-forms, provided $\mathcal D(v) \neq \emptyset$. \smallskip

Similarly, for a given non-vanishing vector field $v$ and a $n$-distribution $\mathcal H$, transversal to $v$, the space  $\mathcal D(v, \mathcal H)$ of $n$-forms $\Theta$, integrally dual to $(v, \mathcal H)$, is convex in the space of all $n$-forms, provided $\mathcal D(v, \mathcal H) \neq \emptyset$.
\end{lemma}

\begin{proof} 
Consider two integrally dual to $v$ forms $\Theta_1, \Theta_2$. 
Evidently, if $\Theta_1$ and $\Theta_2$ are closed, so is their linear combination, the form $\Theta = t\Theta_1 + (1-t)\Theta_2$, where $t \in [0,1]$. Also, if $v\in \ker(\Theta_1)$ and $v \in \ker(\Theta_2)$, then $v \in \ker(\Theta)$. The positivity condition $\pm \Theta|_{\d^\pm_1X(v)} \geq 0$ follows, since $t \geq 0$ and $1-t \geq 0$. 
If $\Theta_1$ and $\Theta_2$ are integrally dual to $(v, \mathcal H)$, then, in addition, $\Theta|_{\mathcal H} > 0$ since  $\Theta_1|_{\mathcal H} > 0$ and $\Theta_2|_{\mathcal H} > 0$.
\end{proof}


\begin{definition}\label{div}  Let $v$ be a non-vanishing vector field on a smooth oriented $(n+1)$-manifold $X$. 
We say that $v$ is \textsf{intrinsically nildivergent} if there exists a volume $(n+1)$-form $\Omega$ on $X$ such that $d(v\,\rfloor \, \Omega) = 0$. \hfill $\diamondsuit$
\end{definition}

 Let $v = \sum_{i=1}^{n+1} a_i \d_{x_i}$ in some local coordinates $(x_1, \dots , x_{n+1})$ on $X$. A volume form $\Omega = f \cdot dx_1 \wedge  \dots \wedge x_{n+1}$ produces nil-divergent form $\Theta = v\rfloor \Omega$,  if the function $f$ satisfies the equation $\langle \nabla f, v \rangle = - \textup{div}(v) f$, where $\langle\;, \;\rangle$ denotes the Euclidean scalar product, $\nabla f$ the gradient of $f$, and  $\textup{div}(v) := \sum_{i=1}^{n+1} \frac{\d a_i}{\d x_i}$.  
\smallskip

Let us describe one mechanism that produces $\Theta$, integrally dual to a given $v$. 

\begin{lemma}\label{nildivergent_form} A non-vanishing vector field $v$ on a $(n+1)$-dimensional $X$ is intrinsically nildivergent with the help of $(n+1)$-form $\Omega$ if and only if $\Omega$ is $v$-invariant. Then the form $\Theta := v\,\rfloor \, \Omega$  is integrally dual to $v$ (see Definition \ref{div}), and $\Theta$ is $v$-invariant as well.
Moreover, if $X$ is oriented, then there is a $n$-distribution $\mathcal H$ on $X$ that is transversal to $v$ and such that $\Theta |_{\mathcal H} > 0$.
\end{lemma}

\begin{proof} Since $d\Omega =0$, we get $\mathcal L_v(\Omega) = d(v\,\rfloor \, \Omega)$. Thus $d(v\,\rfloor \, \Omega) =0$ if and only if $\Omega$ is $v$-invariant.
By Definition \ref{div}, an intrinsically nildivergent form $\Theta := v\,\rfloor \, \Omega$ is closed.

 Since  $v\,\rfloor \,\Theta := v\,\rfloor \,(v\,\rfloor \, \Omega) =0$, we get $v \in \ker(\Theta)$. Moreover, since $\Omega$ is a volume form, $\dim(\ker(\Theta)) =1$. As in Lemma \ref{lem2.1}, it follows that $\Theta$ is a $v$-invariant form. 
 
Let us pick a metric $g$ on $X$ so that $\Omega$ is its volume form. Let $\nu$ be the unit vector field, inward normal in the metric $g$ to $\d X$ in $X$. Then the quotient  $(v\rfloor \, \Omega)|_{\d X}/(\nu  \rfloor \, \Omega)|_{\d X}$  of the two $n$-forms, being restricted to $\d X$, equals to $\cos(\angle_g(v, \nu))$, a non-negative function on $\d_1^+X(v)$ and non-positive on $\d_1^-X(v)$ by the very definition of these two loci. The $g$-orthogonal to $v$ subbundle $v^\perp_g \subset T_\ast X$ is the desired distribution $\mathcal H$.
\end{proof}

\begin{lemma}\label{lem2.3} 
A diffeomorphism $\phi: X \to X$ transforms any form $\Theta$, intergrally dual to a given vector field $v$, into the form $\phi^\ast(\Theta)$, intergrally dual to $\phi^{-1}_\ast(v)$.

In particular, if $v$ is nildivergent with the help of a $(n+1)$-volume form $\Omega$, and a diffeomorphism $\phi$ is such that $\phi^\ast(\Omega) = \pm \Omega$, then the vector field $\phi_\ast(v)$ is nil-divergent. 
\end{lemma}

\begin{proof}  If $v\, \rfloor \, \Theta = 0$, then by naturality, $\phi^{-1}_\ast(v)\, \rfloor \,\phi^\ast(\Theta) = 0$. 
Also by naturality,  $d(\phi^\ast(\Theta)) = \phi^\ast(d\Theta) = 0$.

Any diffeomorphism $\phi$ maps $\d_1^\pm X(v)$ to $\d_1^\pm X(\phi_\ast(v))$.  On the other hand, if $\pm\Theta(w) \geq 0$ for a polyvector $w \in \bigwedge^n T_\ast(\d_1^\pm X(v))$, then $\pm\phi^\ast(\Theta)(\phi^{-1}_\ast(w)) \geq 0$. Thus, if $\pm\Theta |_{\d_1^\pm X(v)} \geq 0$, then $\pm \phi^\ast(\Theta) |_{\d_1^\pm X(\phi^{-1}_\ast(v))} \geq 0$. 
\smallskip

If $d(v \, \rfloor \, \Omega) = 0$ and $\phi^\ast(\Omega) = \pm\Omega$, then $d(\phi^{-1}_\ast(v) \, \rfloor \,\phi^\ast(\Omega)) = d(v  \, \rfloor \, \pm \Omega) = 0.$ So $\phi^{-1}_\ast(v)$ is nildivergent, provided that $v$ is and $\phi$ preserves, up to a sign,  the volume form. Hence the group of volume-preserving/reversing diffeomorphisms $\mathsf{Diff}(X, \Omega)$ of $X$ acts naturally on the space of nildivergent vector fields.
\end{proof}

\section{On proto-billiards maps and Poincar\'{e} return maps}

In order to introduce some \emph{dynamics} in our discussion of the causality maps of traversing flows, we will need to assume the validity of the following property.  
\smallskip

\noindent $\bullet$ {\bf The Involution Hypotheses.} \emph{Let $X$ be a compact connected smooth $(n+1)$-manifold $X$ with boundary. For a traversing boundary generic vector field $v$ and a differential $n$-form $\Theta$, integrally dual on $X$ to $v$ (as in Definition  \ref{def.dual}), we assume that there exists a diffeomorphism $\tau: \d_1^-X(v) \to \d_1^+X(v)$ such that 
\begin{eqnarray}\label{tau}
\tau^\ast\big(\Theta|_{\d_1^+X(v)}\big) = \Theta|_{\d_1^-X(v)}.
 \end{eqnarray}
 }
 %
\indent Occasionally, we will assume that $\tau$ is the restriction to $\d_1^+X(v)$ of a smooth involution $\hat\tau: \d X \to \d X$, whose fixed point set is the locus $\d_2X(v) := \d(\d_1^+X(v)) = \d(\d_1^-X(v))$. In the case of billiard maps on a Riemannian manifold $M$, $\tau$ is a smooth involution $\d(SM) \to \d(SM)$, induced by the ellastic reflection of tangent vectors from $TM |_{\d M}$ with respect to the boundary $\d M$.
\smallskip
 
Of course, in general, such $\tau: \d_1^-X(v) \to \d_1^+X(v)$ may be unavailable (in particular, when $\d_1^+X(v)$ and $\d_1^-X(v)$ are not diffeomorphic)! However, assuming its existence, the \textsf{proto-billiard map}
\begin{eqnarray}\label{proto_billiard}
B_{v, \tau} := \tau \circ C_v:\; \d_1^+X(v) \to \d_1^+X(v),
\end{eqnarray}
the composition of the causality map $C_v$ with the diffeomorphism $\tau$, by Proposition \ref{prop.C_v_preserves}, preserves the measure $\mu_\Theta$. As a result, for such $\tau$ and $\Theta$, it is possible to talk about the \emph{dynamics} of $\mu_\Theta$-preserving iterations $\{(B_{v,\, \tau})^{\circ k}\}_{k\in \Z_+}$ of $B_{v, \tau}$.\smallskip

In what follows, we say that a property is valid \textsf{almost everywhere}, if it may be violated only for the set of points of zero measure.

\begin{theorem}\label{proto-billiard} Let $X$ be a compact smooth $(n+1)$-manifold, equipped with a traversing and boundary generic vector field $v$ and an integrally dual to it $n$-form $\Theta$. 
Let $\tau: \d_1^-X(v) \to \d_1^+ X(v)$ be a diffeomorphism which satisfies the Involution Hypotheses (\ref{tau}). \smallskip

Then the proto-billiard map $B_{v, \tau}: \d_1^+X(v) \to \d_1^+X(v)$ from (\ref{proto_billiard}) has the  \textsf{infinite return property}:  any point $x \in \d_1^+X(v)$ has an open neighborhood $U \subset \d_1^+X(v)$ such that, for almost every $x' \in U$, the sets $\{(B_{v, \tau})^{\circ k}(x')\}_k$ intersect $U$ for infinitely many $k$'s. \smallskip

Let $N$ be an open neighborhood of $\d_1X$ in $X$. In particular, the return property holds for an intrinsically nildivergent 
traversing and boundary generic $v$ and for $\tau$, induced by a diffeomorphism $\tilde\tau : N \to N$ such that $\tilde\tau^\ast(\Omega) = -\Omega$ and $\tilde\tau_\ast(v) = -v$, where $\Omega$ is a volume form.
\end{theorem} 

\begin{proof} Combining Proposition \ref{prop.C_v_preserves} with the property 
$\tau^\ast(\Theta|_{\d_1^+X(v)}) = \Theta|_{\d_1^-X(v)}$,
we conclude that  the  proto-billiard map $B_{v, \tau}$ preserves the measure  $\mu_\Theta$ on $\d_1^+X(v)$. Since the volume $\mu_\Theta(\d_1^+X(v))$ is finite, the standard Poincar\'{e} argument (see \cite{W}) 
about  measure-preserving transformations leads to the infinite return property. 
\smallskip

In the case of nildivergent $v$, we have $\Theta := v \rfloor \Omega$, where $\Omega$ is a $v$-invariant volume form.  Assuming that a diffeomorphism $\tau : \d_1^+X(v) \to \d_1^-X(v)$ admits a lifting $\tilde\tau: N \to N$ so that $\tilde\tau^\ast(\Omega) = -\Omega$ and $\tilde\tau_\ast(v) = -v$ and employing Lemma \ref{lem2.3}, we get 
$$\tilde\tau^\ast(\Theta) = \tilde\tau^\ast(v\; \rfloor \;\Omega) = (-v)\; \rfloor\; (-\Omega) = v\; \rfloor \; \Omega = \Theta$$ 
in $N$. So $B_{v, \tau}$ preserves the measure  $\mu_\Theta$ on $\d_1^+X(v)$, and hence the return property follows.
\end{proof}

In the sections to come, we will strive to construct involutions $\tau$, subject to (\ref{tau}), for the geodesic vector field $v^g$ on the tangent spherical fibration $SM \to M$ with a compact Riemannian manifold $(M, g)$ for the base. 
\smallskip

\begin{definition}\label{gradient_field}
 A smooth vector field $v$ on $X$ is {\sf gradient-like}, if there are a smooth function $f: X \to \R$ and a Riemannian metric $g$ in the vicinity of the zero locus $Z(v)$ of $v$ so that: 
{\bf (1)} $df(v) > 0$ in $X\setminus Z(v)$, and {\bf (2)} $v$ is the gradient $\nabla_g(f)$ in the vicinity of $Z(v)$.  
\hfill $\diamondsuit$
\end{definition}

\begin{definition}\label{Lyap}
Given a smooth vector field $v$ on a compact manifold $X$,  consider an open finite cover $\mathcal U =\{U_i\}$ of $X$ such that,   in each $U_i$, $v$ admits a smooth Lypunov function $f_i: U_i \to \R$, satisfying the following constraints: 
$df_i(v) \geq 0 \text{ in } U_i, \text{ and } df_i(v) > 0 \text{ in } \hfill \break U_i \setminus (U_i \cap Z(v)).$
\smallskip

We call the minimal cardinality of such covers $\mathcal U$ the {\sf Lyapunov genus} of $v$ and denote it by $\mathsf{Lyap}(v)$. \hfill $\diamondsuit$
\end{definition}

By definition, for any gradient-like vector field $v$, $\mathsf{Lyap}(v) =1$. 



\begin{lemma} Let $X$ be a compact manifold, and $v$ a non-vanishing vector field on $X$. If a smooth hypersurface $H \subset X$ bounds a domain $X_1 \subset X$ so that, for any $v$-trajectory $\g$, the connected components of $\g \cap X_1$ and of $\g \cap (X \setminus \mathsf{int}(X_1))$ are singletons or closed intervals, then $\mathsf{Lyap}(v) \leq 2$.
\end{lemma}

\begin{proof} If the connected components of $\g \cap X_1$ and $\g \cap X \setminus \mathsf{int}(X_1)$ are singletons or closed intervals, then, by Lemma 4.1 from \cite{K1}, the vector fields $v|_{X_1}$ and $v|_{X_2}$ are of the gradient type and thus each vector field admits a Lyapunov function.  So we get $\mathsf{Lyap}(v) \leq 2$.
\end{proof}

Thus, to show that $\mathsf{Lyap}(v) \leq 2$ for a non-vanishing $v$, it would suffice to find a separating hypersurface $H \subset X$ that chops all the $v$-trajectories $\g$ into either closed segments $\{\g_\a \subset X_1\}_\a$ and $\{\g_\b \subset X_2\}_\b$, whose boundaries reside in $H$, or into isolated singletons (produced by connected components of the loci where $\g$ is tangent to $H$). For example, if $Z(v)$ and the set $C(v)$ of closed $v$-trajectories are finite, and any trajectory that is not homeomorphic to a closed segment asymptotically approaches $Z(v) \cup C(v)$ at least in one direction, then such a chopping hypersurface $H$ exists. 
\smallskip




For a non-vanishing $v$ with $\mathsf{Lyap}(v) = 2$, let us consider the following construction that ``substitutes" for the desired proto-billiard map $\tau:\d_1^-X(v) \to \d_1^+X(v)$.
\smallskip



Let $Y$ be a \emph{closed} smooth $(n+1)$-dimensional manifold, equipped with a volume $(n+1)$-form $\Omega$, 
and a non-vanishing vector field $v$. All these structures on $Y$ are presumed to be smooth. In addition, assume that the following properties hold:
\begin{eqnarray}\label{double1}
\end{eqnarray}
\begin{itemize}
\item $Y$ is a union of two compact manifolds, $X_1$ and $X_2$, that share a smooth boundary $\d X_1 = \d X_2$ and such that $\mathsf{int}(X_1) \cap \mathsf{int}(X_2) = \emptyset$,
\item $v$ is boundary generic with respect to the hypersurface $\d X_1 \subset Y$,
\item For $i=1,2$, the restriction of $v$ to $X_i$ admits a Lyapunov function $f_i: X_i \to \R$, 
\item $v$ is nildivergent on $Y$, i.e., $d(v \, \rfloor \, \Omega) = 0$ for a volume form $\Omega$ on $Y$.
\end{itemize}

Thanks to the existence of $f_i: X_i \to \R$ and the third bullet in (\ref{double1}), the vector field $v_i := v|_{X_i}$ is  traversing and boundary generic on $X_i$. 

Note that $\d_1^\pm X_1(v_1) = \d_1^\mp X_2(v_2)$. Thus we have two causality maps:
\begin{eqnarray}\label{two_caus}
C_{v_1}: \d_1^+ X_1(v_1) \to \d_1^- X_1(v_1) = \d_1^+ X_2(v_2) \;\; \text{and} \nonumber \\ C_{v_2}: \d_1^+ X_2(v_2) \to \d_1^- X_2(v_2) = \d_1^+ X_1(v_1).
\end{eqnarray}

Their composition produces the \textsf{Poincar\'{e} return map}
\begin{eqnarray}\label{Poincare}
P_v := C_{v_2} \circ C_{v_1}:\;  \d_1^+ X_1(v_1) \to  \d_1^+ X_1(v_1).
\end{eqnarray}

Note that the locus $\d_2^- X_1(v_1) = \d_2^+ X_2(v_2)$ is the fixed point set of the map $C_{v_1}$, but not of $C_{v_2}$; similarly, the locus $\d_2^+ X_1(v_1) = \d_2^- X_2(v_2)$ is the fixed point set of the map $C_{v_2}$, but not of $C_{v_1}$. 
Note also that if $v_1$ is concave in $X_1$, then $v_2$ is convex  in $X_2$; in such a case, $C_{v_1}$ is discontinuous, but $C_{v_2}$ is continuous. Thus, $C_{v_2}$ may play the role of $\tau$. 
\smallskip

\begin{corollary} Under the hypotheses (\ref{double1}), the Poincar\'{e} return map $P_v$ preserves the measure $\mu_\Theta$ on $\d_1^+ X_1(v_1) = \d_1^- X_1(v_2)$, induced by the closed $v$-invariant $n$-form $\Theta :=v \, \rfloor \, \Omega$. 

As a result, any point $x \in \d_1^+X(v_1)$ has an open neighborhood $U \subset \d_1^+X(v_1)$ such that, for almost every $x' \in U$, the sets $\{(P_v)^k(x')\}_k$ intersect $U$ for infinitely many $k$'s.
\end{corollary}

\begin{proof} Under the hypotheses (\ref{double1}), the closed $n$-form $\Theta :=v \, \rfloor \, \Omega$ has all the properties, listed in Lemma \ref{lem2.1}, on both manifolds, $X_1$ and $X_2$. By this lemma, both maps, $C_{v_1}$ and $C_{v_2}$, preserve the measure $\mu_\Theta$, induced by the restriction of $\Theta$ to $\d X$, and so does their composition, the Poincar\'{e} return map $P_v$.
\end{proof}
\smallskip


\noindent{\bf Remark 3.1.} Let $\mathcal F(v)$ and $\mathcal F(v_1), \mathcal F(v_2)$ be the $1$-dimensional oriented foliations, produced by the vector fields $v$ and $v_1 := v|_{X_1}, v_2 := v|_{X_2}$, respectively.  Although, according to Holography Theorem \ref{th1.1}, the map $C_{v_i}$ ($i=1,2$) allows for a reconstruction of the topological type of the pair $(X_i, \mathcal F(v_i))$, the Poincar\'{e} return map $P_v$ in (\ref{Poincare}) \emph{alone}  seems to be insufficient for a reconstruction of the pair $(Y, \mathcal F(v))$. 

Also, while each trajectory space $\mathcal T(v_i) := \d_1^+ X_1(v_i)\big/ \{x \sim  C_{v_i}(x)\}$ is a separable compact space (actually, a $CW$-complex), the quotient $\mathcal T(v) := \d_1^+ X_1(v_i)\big/ \{x \sim  P_v(x)\}$ typically is pathological (non-separable).
\hfill $\diamondsuit$
\begin{theorem}{\bf (Double Holography)}\label{double_holography} 
Under the hypotheses (\ref{double1}), the two causality maps $C_{v_1}$ and $C_{v_2}$ from (\ref{two_caus})
are sufficient for a reconstruction of the topological type of the pair $(Y, \mathcal F(v))$.

If each $v$-trajectory hits the locus $\d X_1 = \d X_2$ transversally at \emph{some} point or is quadratically tangent to it at \emph{some} point, and 
$df_1 = df_2$ in the vicinity of $\d X_1$, 
then $C_{v_1}$ and $C_{v_2}$ are sufficient for a reconstruction of the \emph{smooth} topological type of the pair $(Y, \mathcal F(v))$.
\end{theorem}

\begin{proof} Let $i = 1, 2$. Note that $C_{v_i}$ allows for a reconstruction of the trajectory space $\mathcal T(v_i)$ as the quotient space $\d X_i\big/ \{x \sim C_{v_i}(x)\; | \; x\in \d_1^+X_i(v_i)\}$. 
This construction produces a continuous map $\Gamma^\d_i: \d X_i \to \mathcal T(v_i)$,  the restriction of the obvious map $\Gamma_i: X_i \to \mathcal T(v_i)$ to the boundary $\d X_i$. 

Let $\mathcal H_i$ denote the $1$-dimensional foliation of $\mathcal T(v_i)\times \R$ by the fibers of the obvious trivial fibration  $\pi_i:  \mathcal T(v_i)\times \R \to  \mathcal T(v_i)$.

We pick a pair Lyapunov functions $f_i: X_i \to \R$ ($i = 1, 2$). They help to realize $(X_i, \mathcal F(v_i))$ as the pull-back of $(\mathcal T(v_i)\times \R,\, \mathcal H_i)$ under the embedding $\b_i: X_i \to \mathcal T(v_i)\times \R$, given by the formula $\b_i(x) = (\Gamma_i(x), f_i(x))$. Note that $\b_i^\d := \b_i|: \d X_i \to \mathcal T(v_i)\times \R$ separates $\mathcal T(v_i)\times \R$ into two regions, so that  the compact region is $\b_i (X_i)$.  Therefore, the knowledge of the imbedding $\b_i^\d$ (which is produced using $C_{v_i}$ and $f_i|_{\d X_i}$ only) is sufficient for a reconstruction of the topological type of the pair $(X_i, \mathcal F(v_i))$ (\cite{K8}, Theorem 4.1). 

Recall that $Y = X_1 \bigcup_{\{\d X_1 = \d X_2\}} X_2$ and that $ \mathcal F(v_1)$ and $\mathcal F(v_2)$ match across $\d X_1 = \d X_2$ to form the smooth foliation $ \mathcal F(v)$.
Now, we glue $\b_1(X_1)$ and $\b_2(X_2)$ via the homeomorphism $\b_2 \circ \b_1^{-1}: \b_1(\d X_1) \to \b_2(\d X_2)$. The result of the gluing is a manifold $Z$, homeomorphic to $Y$.
The two $v$-oriented foliations  $\mathcal H_1 = \b_1(\mathcal F(v_1))$ and $\mathcal H_2 = \b_2(\mathcal F(v_2))$ match continuously (in fact, piecewise differentiably) in the vicinity of $\d X_1= \d X_2$, thus forming topological $1$-dimensional foliation on $Y$, which is homeomorphic to $\mathcal F(v)$. 

If any $v$-trajectory is \emph{somewhere} transversal to $\d X_1 = \d X_2$ or somewhere quadratically tangent to it, then by an argument as in \cite{K4}, Theorem 4.1 and Lemma 4.3, these data are sufficient for reconstructing the \emph{smooth} topological types of $(X_i, \mathcal F(v_i))$ (basically, since the solutions of an ODE depend smoothly on the initial data). Assuming that $df_1 = df_2$ in the vicinity of $\d X_1$ (this is hypotheses is restrictive: it implies the existence of a closed $1$-form $\a$ on $Y$ such that $\a(v) > 0$), the gluing map is the identity on $\d X_1$ and on the tangent bundle $TX|_{\d X_1}$. As a result, the foliations  $\mathcal F(v_1)$ and $\mathcal F(v_2)$ match differentiably across $\d X_1$. 
\end{proof}





\section{On the $v$-harmonizing metrics and the associated minimal foliations}
%
Any closed and co-closed nonsingular 1-form $\a$ on a compact Riemannian  $(n+1)$-manifold $X$ produces a beautiful geometric structure: a pair of mutually orthogonal foliations $\mathcal F_\a$ and $\mathcal G_\a$ of dimensions 1 and $n$, respectively,  both of which are {\sf minimal} \cite{K6}, \cite{Su}, provided that the form $\a$ has the following global property:
 \smallskip

\noindent $\bullet$ {\bf The Calabi Condition \cite{Ca}:} 
\textsf{Through each point $x \in X$, there exists a smooth path $\g$ such that either $\g$ is a loop or a segment, with its ends residing in $\d X$, and}  
\begin{eqnarray}\label{Calabi}
\a(\dot\g) > 0. 
\end{eqnarray}

If a closed $1$-form $\a$ satisfies (\ref{Calabi}), it is called {\sf transitive} \cite{Ca}. In fact, in \cite{Ca}, Calabi studied closed $1$-forms $\a$ that may have Morse type singularities (different from extrema) and are transitive. He proves that such transitive $\a$ is {\sf intrinsically harmonic}, i.e., there exists a Riemannian metric $g$ so that  $\a$ is also co-closed in $g$. Moreover, the transitivity of $\a$ is also necessary for its intrinsic harmonicity.
\smallskip

Throughout this paper, we embed properly a given compact $(n+1)$-dimensional manifold $X$ into an open manifold $\hat X$ (when $X$ is closed, $\hat X = X$) and extend $v$ to a non-vanishing vector field $\hat v$. We treat $(\hat X, \hat v)$ as a ``germ" surrounding  $(X, v)$.\smallskip

Let $\theta$ be a closed $1$ form on $S^1 \times D^n$, the pull-back of the canonical $1$-form on $S^1$ under the obvious projection $S^1 \times D^n \to S^1$. 
Similarly, consider the obvious function $D^1 \times D^n \to D^1 \subset \R$ and denote by $\theta$ the differential of this function. 

\begin{definition}\label{Calabi_tube}{\bf (The Calabi tubes)}
Assume that a compact oriented manifold $X$ is equipped with a non-vanishing vector field $v$. 
\begin{itemize}
\item We call an orientation-preserving  diffeomorphism $\kappa: S^1 \times D^n \to \hat X$ or $\kappa: D^1 \times D^n \to \hat X$ a {\sf Calabi tube}, if $\theta(\kappa^{-1}_\ast(v)) > 0$. 
\smallskip

\item We call a Calabi tube $v$-{\sf invariant}, if the orientation-preserving  diffeomorphisms $\kappa: S^1 \times D^n \to \hat X$  or $\kappa: D^1 \times D^n \to \hat X$ are such that $v$ is tangent to every curve $\{\kappa(S^1 \times u)\}_{u \in D^n}$ or to every curve $\{\kappa(D^1\times u)\}_{u \in D^n}$, respectively. \smallskip 

\item We call a $v$-invariant Calabi tube {\sf balanced} if the function $(\kappa^{-1})^\ast(\theta)(v)$ is constant along each $v$-trajectory. 
\hfill $\diamondsuit$
\end{itemize}
\end{definition}


\begin{definition}\label{Calabi_field} \noindent {\bf (The Calabi vector fields)} 
\begin{itemize}
\item We say that a non-vanishing vector field $v$ on $X$ is a \textsf{Calabi field}, if $X$ admits a cover by Calabi tubes (whose images reside in $\hat X$).\smallskip

\item We say that a non-vanishing vector field $v$ on $X$ is an \textsf{invariant Calabi field}, if $X$ admits a cover by $v$-invariant Calabi tubes.\smallskip

\item We say that an invariant Calabi vector field $v$ on $X$ is an \textsf{balanced}, if $X$ admits a cover by balanced $v$-invariant Calabi tubes.
\hfill $\diamondsuit$
\end{itemize}
\end{definition}

\noindent {\bf Remark 4.1.} It follows from Lemma \ref{travesing_is_Calabi} below that if $v$ is a Calabi field, admitting a cover by $a$ toroidal and $b$ cylindrical 
$v$-invariant Calabi tubes, then $\mathsf{Lyap}(v) \leq 2a +b$.
\hfill $\diamondsuit$
\smallskip

\begin{lemma}\label{invariant_versa_balanced} If $v$ is an invariant Calabi field with respect to a cover of $X$ by Calabi tubes, 
then $v$ is a balanced invariant Calabi field (with respect to a differently parametrized Calabi cover). So the second and third bullets in Definition \ref{Calabi_field}  are equivalent requirements. 
\end{lemma}

\begin{proof} Let $x: D^1 \times D^n \to D^1$ and $y: D^1 \times D^n \to D^n$ be the obvious coordinates on the product.  Let $\kappa^{-1}: U_{\hat\g^\star} \to D^1 \times D^n$ be as in Definition \ref{Calabi_tube}, second bullet, and put $\theta = dx$. 
Our goal is to construct a new diffeomorphism $\tilde\kappa^{-1}: U_{\hat\g^\star} \to D^1 \times D^n$ so that the image of each $v$-trajectory $\g \subset U_{\hat\g^\star}$ is still the fiber of the projection $y: D^1 \times D^n \to D^n$ and the function $[(\tilde\kappa^{-1})^\ast\theta](v)$ is constant on $\g$. 

Let $\kappa^{-1}_\ast(v) = f(x,y)\,\d_x$. So $\theta(\kappa^{-1}_\ast(v)) = dx(\kappa^{-1}_\ast(v))$ is the function $f(x, y) > 0$. Consider the auxiliary function $$g(x, y) = \int_0^x \frac{dt}{f(t, y)}\Big/ \int_0^1 \frac{dt}{f(t, y)},$$ which is strictly increasing in $x \in [0, 1]$ and has the property $g(0, y) =0,\, g(1, y) =1$. We define the diffeomorphism $\phi: D^1 \times D^n \to D^1 \times D^n$ by the formula  $(x', y') := (g(x, y), y)$. Then the $\phi$ transforms $f(x, y)\,\d_x$ into $\d_{x'}$.
\smallskip

The case of a toroidal Calabi tube $U_{\hat\g^\star}$ is similar: $x: S^1 \times D^n \to S^1$ is a circular-valued map, viewed as a function with the period $1$, and $f(x+1, y) = f(x, y)$ for all $x, y$. Under these assumptions, the same formulas deliver the desired diffeomorphism $\phi: S^1 \times D^n \to S^1 \times D^n$ such that $[(\tilde\kappa^{-1})^\ast\theta](v) = 1$.
\end{proof}

\begin{lemma}\label{travesing_is_Calabi} Any traversing vector field $v$ is an invariant balanced Calabi field. Also, $\mathsf{Lyap}(v) =1$.
\end{lemma}

\begin{proof} By the definition of a traversing vector field (see \cite{K1}), each $v$-trajectory $\g$ is a closed segment or a singleton. We consider a closed segment $\hat \g \subset \hat X$ of the $\hat v$-trajectory that contains $\g$ in its interior and such that $\d\hat\g \subset \hat X \setminus X$. We take a disk-shaped smooth transversal section $D^n$ of the $\hat v$-flow at a point $o \in \g$ and form a small $\hat v$-invariant tubular neighborhood $U_{\hat\g}$ of $\hat\g$ in $\hat X$ by taking the union of $\hat v$-trajectories through $D^n$ and ``trimming" this sheaf, as described below. We denote by $V_{\hat\g}$ the trimmed set of $\hat v$-trajectories that pass through the sphere $\d D^n$.

Since $\d\hat\g \cap X = \emptyset$, we may pick the tube $U_{\hat\g} \supset \hat\g$ so narrow (equivalently, the section $D^n$ so small), that $X \cap (\delta U_{\hat\g}) = \emptyset$, where $\delta U_{\hat\g} := \d(U_{\hat\g}) \setminus V_{\hat\g}$ . This choice helps us to introduce a product structure $\kappa_{\hat\g}: I \times D^n \approx U_{\hat\g}$ in $U_{\hat\g}$ so that: 

{(1)} $\kappa_{\hat\g}(I) \times o = \hat\g$, where the $\kappa_{\hat\g}|_I$ is an orientation-preserving diffeomorphism, 

{(2)} $U_{\hat \g}$ consists of segments of $\hat v$-trajectories,  

{(3)} each slice $\kappa_{\hat\g}(t \times D^n)$, where $t \in I$, is transversal to the $\hat v$-flow, and 

{(4)} $\kappa_{\hat\g}$ is an orientation-preserving diffeomorphism with respect to the orientation of $U_{\hat\g}$, induced by the preferred orientation of $\hat X$. 

This choice of the tube $U_{\hat\g}$ satisfies all the properties, listed in Definition \ref{Calabi_tube}, second bullet. So $U_{\hat\g}$ is a $v$-invariant Calabi tube, which contains $\g$. By Lemma \ref{invariant_versa_balanced}, there is a reparametrization of $U_{\hat\g}$ so that $U_{\hat\g}$ becomes balanced.

By the compactness of $X$, it admits a finite subcover by $v$-invariant balanced Calabi tubes $\{U_{\hat\g}\}_{\hat\g}$. So $v$ is an invariant balanced Calabi field.
\smallskip

Since $v$ admits a global Lyapunov function \cite{K1}, we get $\mathsf{Lyap}(v) =1$.
\end{proof}


Given a Riemmanian metric $g$ on a $(n+1)$-dimensional $X$, we consider the {\sf Hodge star operator} $\ast_g: T^\ast X \to \bigwedge^n T^\ast X$ which is a bundle isomorphism. In local coordinates, and with respect to a local basis $\{e_j^\star\}_j$ in $T^\ast X$ and some associated dual local basis $\{\eta_j^\star\}_j$ in $\bigwedge^n T^\ast X$, the operator $\ast_g$ is given by the $(n+1)\times (n+1)$-matrix $\mathsf G =  \sqrt{\det(\mathsf g)}\cdot (\mathsf g)^{-1}$, where $\mathsf g = (g_{jk})$.  For $n \geq 2$, since $\det(\mathsf G) = \det(\mathsf g)^{(n-1)/2}$, remarkably,  the operator $\ast_g$ determines the metric $g$ (\cite{Ca}). 
\smallskip

Recall that the {\sf co-derivative operator},  acting on differential $p$-forms on $X$, is defined by 
 $$\delta =_{\mathsf{def}} (-1)^{(n+1)(p+1)+1}\, (\ast_g) \circ\, d\,\circ (\ast_g).$$ We say that a $p$-form $\a$ is {\sf co-closed} if $\delta\a = 0$. The closed and co-closed forms $\a$ are {\sf harmonic}, i.e., they satisfy the Laplace equation $(d + \delta)^2\a = 0$; however, on manifolds with boundary, not any harmonic form is closed and co-closed! 
\smallskip

\noindent {\bf Remark 4.1.} Given a $1$-form $\a$ on $(X, g)$, we denote by $\a_{\mathsf{tan}}$ the $1$-form in $TX|_{\d X}$ that coincides with $\a$ on $T(\d X)$ and vanishes on the normal vector field $\nu_g(\d X, X)$. By definition, $\a_{\mathsf{norm}} = \a - \a_{\mathsf{tan}}$, as sections of  $TX|_{\d X}$. 
\smallskip

On manifolds with boundary, the basic relation between closed and co-closed forms and the DeRham cohomology is a bit subtle, as described in \cite{CTGM}; for $1$-form $\a$ on $X$, the relation is given by  
 $$H^1(X; \R) \approx \{\a|\; d\a = 0,\; \delta\a = 0,\; \a_{\mathsf{norm}} = 0\},$$ 
 $$H^1(X; \d X; \R) \approx \{\a|\; d\a = 0,\; \delta\a = 0,\; \a_{\mathsf{tan}} = 0\}. $$
 \hfill $\diamondsuit$
\smallskip

The main ideas for proving the next theorem can be found in \cite{K6}, as a special case of Theorem C. However, in Theorem \ref{th2.1} below, the given ingredient is the vector field $v$, not a closed 1-form $\a$ as in \cite{K6}. 

\begin{theorem}\label{th2.1} Let $v$ be a Calabi vector field (see Definition \ref{Calabi_field}, the $1^{st}$ bullet) on a compact $(n+1)$-manifold $X$, $n \geq 2$.\smallskip

{$(1)$} Then there exists a smooth $1$-form $\a$ and a metric $g$ on $X$  such that: 
\begin{itemize}
\item $\a(v) > 0$
\item $d(\ast_g \a) = 0$, 
\item $\dim(\ker(\ast_g \a)) = 1$, 
\item $\a \wedge \ast_g \a = \ast_g(1)$ is a 
volume form on $X$.
\end{itemize}
\smallskip

{$(2)$} Let $v$ be a $v$-\emph{invariant} Calabi vector field 
on $X$. Then, in addition to the properties in {$(1)$}, one may choose $\a$ to be a $v$-invariant $1$-form. The $n$-form $\Theta := \ast_g \a$ is integrally dual to $v$ (see Definition \ref{def.dual}), 
 i.e., in addition to the bulleted properties above, $v \in \ker(\Theta)$. Moreover, one may choose the metric $g$ on $X$ to be $v$-invariant.
\smallskip

{$(3)$} If $v$ is a \emph{traversing} vector field, then $v$ is a balanced $v$-invariant Calabi field.  Moreover, $\a = df$, an exact $1$-form. For an appropriate $g$, the harmonic $n$-form $\Theta = \ast_g \a$ is integrally dual of $v$, and $f: X \to \R$ is a harmonic Lyapunov function for $v$.  Furthermore, one may choose the metric $g$ on $X$ to be $v$-invariant.
\end{theorem}

\begin{proof} Here is a general plan for proving the theorem: {\bf (i)} starting with the Calabi vector field $v$, we construct $v$-amenable $1$-form $\a$ (in the third bullet, $\a = df$, where $f$ is a Lyapunov function for $v$) and a closed $n$-form $\Theta$ so that $\a \wedge  \Theta > 0$; {\bf (ii)} then we construct the metric $g$ for which $\ast_g(\a) = \Theta$.  \smallskip

We consider a cover $\hat{\mathcal U}$ of $X$ by Calabi tubes $U_{\hat \g}$, where $\hat\g$ runs over the set of all $\hat v$-trajectories that have a nonempty intersection with $X$. Using compactness of $X$, we pick a finite subcover  $\{U_{\hat\g_i}\}_{i}$ of $\hat{\mathcal U}$ so that  $X = \bigcup_i\; (U_{\hat\g_i} \cap X)$.\smallskip

As before, we divide  Calabi tubes into two types: for the first type, the core $\hat\g$ of $U_{\hat \g}$ is a closed segment, for the second type, the core is a simple loop. For each of the tubes, we fix a product structure, given by a diffeomorphism $\kappa_{\hat \g}$ with the properties as in Definition \ref{Calabi_tube}, the first bullet.

For each tube of the first type, we consider a function $\tilde f_{\hat\g} : U_{\hat \g} \to I$, the pull-back by $\kappa^{-1}_{\hat\g}$ of the obvious function $I \times D^n \to I$. By the definition of the Calabi tube, $d \tilde f_{\hat\g}(v) > 0$ in $U_{\hat \g}$.
Similarly, for each tube $U_{\hat \g}$ of the second type, with the help of $\kappa^{-1}_{\hat\g}$, we produce a $1$-form $\tilde\a_{\hat\g}$ in $U_{\hat \g}$ such that $\tilde\a_{\hat\g}(v) > 0$.
 
Let $\tilde\psi: D^n \to \R_+$ be a smooth non-negative bell function with the support in the interior of $D^n$ and such that all its partial derivatives vanish along $\d D^n$. Using the $(\kappa^{-1}_{\hat\g})$-induced projection $p: U_{\hat \g} \to D^n$, we form the pull-back function $\psi = \tilde\psi \circ p$ and multiply $\tilde f_{\hat\g}$ by $\psi$ to get a smooth function $f_{\hat\g}: U_{\hat \g} \to \R$ with the support in the interior of $U_{\hat \g}$. 
Thus $df_{\hat\g}(\hat v) > 0$ in the interior of $U_{\hat\g}$ and $df_{\hat\g}(\hat v) \geq 0$ globally. \smallskip


Similarly, for Calabi tubes of the toroidal kind, we put $\a_{\hat\g} := \psi \cdot \tilde\a_{\hat\g}$, where $\psi: U_{\hat\g} \to \R_+$ is the pull-back of the bell function $\tilde\psi: D^n \to \R_+$ under the $(\kappa^{-1}_{\hat\g})$-induced projection $U_{\hat\g} \to D^n$. Thanks to the choice of $\tilde\psi$, this $1$-form is well-defined globally. Again, $\a_{\hat\g}(\hat v) > 0$ in the interior of $U_{\hat\g}$ and $\a_{\hat\g}(\hat v) \geq 0$ globally. Unfortunately, this $\a_{\hat\g}$ is not closed! 
\smallskip

For Calabi tubes of both kinds, we introduce the $n$-form $\Theta_{\hat\g}$ on $U_{\hat\g}$ as the pull-back, under the map $(\kappa_\g)^{-1}$, of the standard volume form $vol_{D^n}$ on $D^n$, being multiplied by the bell function $\tilde\psi: D^n \to \R_+$. Evidently, $\Theta_{\hat\g}$ extends trivially on $X$. Since $\Theta_{\hat\g}$ depends only on the coordinates in $D^n$, for the dimensional reason, we get $d\Theta_{\hat\g} = 0$.  

Moreover, the restriction  of the function $\pm\Theta_{\hat\g} /\Omega^\d$ to $\d_1^\pm X(v) \setminus \d_2X(v)$ is positive.  Indeed, consider the subtube $U_{\hat \g}^\dagger \subset U_{\hat\g}$ that is the preimage of $\g \subset \hat\g$ under the projection $U_{\hat \g} \to \hat\g$, delivered by the product structure on the tube $U_{\hat\g}$.
Then, in the case of $\d_1^- X(v) \setminus \d_2X(v)$ (or of $\d_1^+ X(v) \setminus \d_2X(v)$), $\hat v$ points outside (inside) of both domains, $X$ and $U_{\hat \g}^\dagger$.  
\smallskip

Next, we define the global $1$-form $\a$ on $X$ by the formula $\sum_i \a_{\hat\g_i}$, and the global $n$-form $\Theta$ by the formula $\sum_i \Theta_{\hat\g_i}$. Since
$d\Theta_{\hat\g_i}=0$ by its construction, $\Theta$ is a closed form. Again, since $\a_{\hat\g_i}(\hat v) > 0$ in the open set $U_{\hat\g_i}$ for all $i$, we conclude that $\a(v) > 0$ in $X$.  \smallskip

Finally, with the candidates $\a$ and $\Theta$ in place, we consider the Hodge star bundle isomorphism $\ast_g: T^\ast X \to \bigwedge^n T^\ast X$, where the Riemmanian metric $g$ on $X$ to be consructed. Recall that, for $n \geq 2$, the operator $\ast_g$ determines the metric $g$ (\cite{Ca}). Therefore, it suffices to pick any $g$ such that $\ast_g(\a) = \Theta$. By the construction of $\a$ and $\Theta$, we get $\a \wedge \ast_g \a > 0$. By Lemma 1 from \cite{Ca}, such a metric $g$ exists. 
Moreover, Calabi's argument (see \cite{Ca}, pages 110-112) insures that $\a \wedge \ast_g \a = vol_g := \ast_g(1)$, the $g$-induced volume $(n+1)$-form on $X$. (For the reader's convenience, we will sketch his argument below.) \smallskip
%

With respect to such a choice of $g$, the form $\Theta$ is closed, its kernel is $1$-dimensional,  $\a(v) > 0$, and $\a \wedge \ast_g \a = vol_g$. So claim {$(1)$} from of the theorem is valid.\smallskip

Note that, for each $v$-\emph{invariant} Calabi tube, $v \in \ker(\Theta_{\hat\g_i})$ and $\mathcal L_v(\Theta_{\hat\g_i}) = 0$ for all $i$. Thus, for a $v$-invariant Calabi vector field $v$, in addition, we get $v \in \ker(\Theta)$ and $\mathcal L_v(\Theta) = 0$. 

Thanks to Lemma \ref{invariant_versa_balanced}, by choosing an appropriate parametrization of the invariant Calabi tubes from the cover $\hat{\mathcal U}$, we may assume that $\a(v)$ is constant along the $v$-trajectories. Therefore $\a$ is a $v$-invariant form (i.e., $\mathcal L_v\a = 0$). \smallskip

Next, we are going to show that we may choose  the metric $g$ in {$(2)$} to be $v$-invariant. To achieve this, we need to revisit the argument in \cite{Ca}. 

We notice that, for a given $v$ and $\a$ and $\Theta$ as above, the choice of $g$ is far from being unique. Using the product structure in an invariant balanced Calabi tube $U_{\hat\g}$, we introduce there local coordinates $\{x_i\}_{i \in [0, n]}$ such that: (i) $dx_0 = \a$, where $x_0: U_{\hat\g} \to S^1$ or $x_0: U_{\hat\g} \to \R^1$, depending on the type of the tube, (ii)  $v = \d_{x_0}$ (using that $U_{\hat\g}$ is balanced), and (iii) $dx_1 \wedge \dots \wedge dx_n = \Theta$.  
Then we define the dual basis in $\bigwedge^n T^\ast U_{\hat\g}$ as $$\eta_0 := dx_1 \wedge \dots \wedge dx_n,\, \eta_1 := - dx_0 \wedge dx_2 \wedge \dots \wedge dx_n,\, \eta_2 := dx_0 \wedge dx_1 \wedge dx_3 \dots \wedge dx_n, \; \dots \, \text{etc.}$$
Again, following \cite{Ca}, for $n \geq 2$, this choice of bases $\{dx_i\} \in T^\ast U_{\hat\g}$ and $\{\eta_i\} \in \bigwedge^n T^\ast U_{\hat\g}$ defines a unique candidate for the local star operator $\ast_{g_{\hat\g}}: T^\ast U_{\hat\g} \to \bigwedge^n T^\ast U_{\hat\g}$ that takes each $dx_i$ to $\phi_{\hat\g} \cdot \eta_i$, where the smooth functions $\{\phi_{\hat\g} : U_{\hat\g} \to \R_+\}_{\hat \g}$ form a finite partition of unity, subordinate to the cover $\hat{\mathcal U}$, and such that $\mathcal L_v(\phi_{\hat\g}) = 0$.  

We notice that all the forms $\{dx_i\}$ and $\{\eta_i\}$ are $v$-invariant. Moreover, since each $\phi_{\hat\g}$ does not depend on $x_0$, we get that $\{\phi_{\hat\g} \cdot \eta_i\}_i$ are $v$-invariant as well. Therefore, the local star operators $\ast_{g_{\hat\g}}$ must be also $v$-invariant. \smallskip

Finally, we pick the $v$-invariant operator $\ast_g := \sum_{\hat\g} \ast_{g_{\hat\g}}$ 
 which has the desired properties: $\ast_g(\a) = \Theta$, $\a \wedge \Theta = vol_g$. Therefore, the corresponding metric $g$ must be invariant as well.

Hence, the claim {$(2)$} is valid.\smallskip

When $v$ is traversing, by Lemma \ref{travesing_is_Calabi}, $X$ admits a cover by $v$-invariant Calabi tubes, which are cylinders only. As a result, $\a = df$, where $f:= \sum_i f_i$. Therefore, $\a$ is exact! Moreover, since $d\a =0$ and $d\Theta = 0$, both $\a$ and $\Theta$ are harmonic in $g$. By the constructions of $\a$ and $\Theta$ above, they satisfy all the the properties from the claim {$(3)$} of this theorem, including the property $\mathcal L_vg = 0$. 
\end{proof}




The properties of forms $\Theta$ and $\a$, listed in Theorem \ref{th2.1}, motivate the following:

\begin{definition}\label{harmonizing pair}  For a non-vanishing vector field $v$ on a smooth compact $(n+1)$-manifold $X$, consider the space $\mathcal Har(v)$ of smooth Riemannian metrics $g$ on $X$, paired with smooth $1$-forms  $\a$, such that: 
\begin{enumerate}
\item  $\a(v) > 0$,
\item $d\a = 0$,
\item the $n$-form $\Theta :=_{\mathsf{def}} \ast_g(\a)$ is closed (and thus $\a$ is \emph{harmonic} in the metric $g$),
\item  $v \in K(\Theta)$, the kernel of $\Theta$,
\item $\a \wedge \Theta$, is the $g$-induced $(n+1)$-volume form on $X$, 
\item the function $\Theta/ \Omega^\d \geq 0$ on $\d_1^+X(v)$ and $\Theta/ \Omega^\d \leq 0$ on $\d_1^-X(v)$, where $\Omega^\d$ denotes a volume $n$-form on $\d X$, consistent with its orientation.
\end{enumerate}

For a given $v$, we say that  $(g, \a)$ is a \textsf{harmonizing pair}, if all the six properties 
are valid.
 
\hfill $\diamondsuit$
\end{definition}

\noindent {\bf Remark 4.2.} Note the main difference between Definition \ref{harmonizing pair} and the list of properties in the claim {$(1)$} of Theorem \ref{th2.1}: namely, the form $\a$ must be closed in Definition \ref{harmonizing pair}.
The form $\Theta = \ast_g\a$ is required to be closed in both Definition \ref{harmonizing pair} and Theorem \ref{th2.1}.\hfill $\diamondsuit$
 \smallskip

Thus, by Theorem \ref{th2.1}, any traversing boundary generic vector field $v$ admits a $v$-harmonizing pair $(g, df)$. 
\smallskip 







The proof of the corollary below can be found in \cite{K6}, as a special case of Theorem C. See also \cite{Su} and the proof of Corollary \ref{Plateau} for a sketch of the argument in Theorem C.

\begin{corollary}\label{foliations} For any non-vanishing vector field $v$ and a $v$-harmonizing pair $(g, \a)$ on $X$ (as in Definition \ref{harmonizing pair}), the following properties hold: 
\begin{itemize}
\item the $1$-foliation $\mathcal F(v)$, determined by $v$, is formed by the \emph{geodesic} curves in $g$,
\smallskip 

\item the $n$-foliation $\mathcal G(\a)$, defined by the closed $1$-form $\a$, consists of leaves $\mathcal L$ that \emph{minimize} the $g$-induced $n$-volume among all sufficiently small perturbations of $\mathcal L$ that are 
 fixed on $\d X$. 
\smallskip

\item the leaves of $\mathcal F(v)$ and of $\mathcal G(\a)$ are mutually orthogonal in $g$.
\hfill $\diamondsuit$
\end{itemize}
\end{corollary}

\begin{corollary}\label{Plateau} Let $v$ be a traversing vector field on compact smooth $(n+1)$-manifold $X$ with boundary, $f: X \to \R$ a Lyapunov function for $v$, and $(g, df)$ a $v$-harmonizing pair. \smallskip

Then the Plateau problem for each of the $(n-1)$-dimensional contours $f^{-1}(c) \cap \d X$ has a smooth solution $f^{-1}(c)$ in $(X, g)$ for any regular value $c$ of $f|_{\d X}$. If $H_n(X; \R) = 0$, then this $n$-volume minimizing solution is unique.  \smallskip

For all sufficiently close (in the $C^\infty$-topology) to $v$ traversing vector fields $\tilde v$ and the corresponding $\tilde v$-harmonizing pairs $(\tilde g, df)$, the Plateau problem for the contour $f^{-1}(c) \cap \d X$ still has a smooth solution.
\end{corollary}

\begin{proof} The proof may be extracted from \cite{K6}, Theorems C and D, which deal with a more general setting than the one required here. Let us sketch their main trust.
Since $f$ is a Lyapunov function for $v$ and $v$ is traversing, $f$ attends is extrema on $\d X$.  For any non-critical for $f|_{\d X}$ value $c$, the leaves $\mathcal G_c := f^{-1}(c)$, where $c \in f(X)$, of the foliation $\mathcal G(df)$ are nonsingular. By the construction of $v$-harmonizing $g$, we get $vol_g(\mathcal G_c) = \int_{\mathcal G_c} \Theta$, where $\Theta = \ast_g(df)$. Consider any small smooth perturbation $\mathcal H$ of the hypersurface $\mathcal G_c$, supported in the interior of $X$ so that $\d\mathcal H = \d \mathcal G_c$. Like $\mathcal G_c$, the hypersurface $\mathcal H$ is transversal to $v$. By the Stokes' theorem, $\int_{\mathcal H} \Theta = \int_{\mathcal G_c} \Theta$ since $d \Theta = 0$ and $\mathcal H$ and $\mathcal G_c$ are cobordant. On the other hand, using orthogonality of $v$ to $\mathcal G$, we conclude that, at each point $x \in \mathcal H$ where $T_x\mathcal H$ is not tangent to the locus $\mathcal G_{f(x)}$, the $n$-volume form $dg|_{T_x\mathcal H} > \Theta|_{T_x\mathcal H}$. Thus $vol_{g|}(\mathcal H) > \int_{\mathcal H} \Theta = \int_{\mathcal G_c} \Theta = vol_{g|}(\mathcal G_c)$, unless $\mathcal H$ is tangent to $\mathcal G$ almost everywhere, in which case, $\mathcal H = \mathcal G_c$. Therefore, $\mathcal G_c$ minimizes the $g$-induced $n$-volume locally, provided that $\mathcal G_c \cap \d X$ is fixed. 

If $H_n(X; \R) = 0$, then $\mathcal H\, \cup -\mathcal G_c$ is a trivial $\R$-cycle for any relative cycle $\mathcal H$ that shares with $\mathcal G_c$ its boundary $\mathcal G_c \cap \d X$. Hence,  similar arguments work for such an $\mathcal H$, namely, $vol_{g|}(\mathcal H) > vol_{g|}(\mathcal G_c)$. Thus the volume minimizing solution $\mathcal H$ of the Plateau problem for the contour $f^{-1}(c) \cap \d X$ is unique. 
\end{proof}

Combining Theorem \ref{th2.1} with Lemma \ref{lem2.1} leads instantly to the following claim. 

\begin{theorem}\label{traversing_harmonizing} For any traversing boundary generic vector field $v$ on a $(n+1)$-manifold $X$, there exists a $v$-\emph{invariant} metric $g$  and a smooth Lyapunov function $f: X \to \R$ such that $(g, df)$ is a $v$-harmonising pair.
\smallskip 

For any such pair $(g, df) \in \mathcal Har(v)$, the measure $\mu_\Theta$ on $\d X$, induced by the closed $n$-form $\Theta =_{\mathsf{def}} \ast_g(df)$ on $X$ via formula (\ref{eq2.2}), is preserved under the causality map $C_v$. The forms $\Theta$ and $df$ are $v$-invariant.
\hfill $\diamondsuit$
\end{theorem}

In combination with Theorem \ref{traversing_harmonizing}, Corollary \ref{foliations} leads to the following claim.


\begin{corollary}\label{from_g_to_foli} Let $v$ be a traversing boundary generic vector field on $X$, and $(g, df)$ a $v$-harmonizing pair. \smallskip
Assume that no pair of distinct points $a, b \in \d X$ admits two distinct $g$-geodesics that reside in $X$ and connect $a$ and $b$.\footnote{For example, this is the case when $(X, g)$ is contained in a larger Riemannian manifold $(\hat X, \hat g)$ such that any two distinct points in $\hat X$ belong to a single $\hat g$-geodesic.} \smallskip

Then the the knowledge of the causality map $C_v: \d_1^+X(v) \to \d_1^-X(v)$ and of the $v$-harmonizing metric $g$ on $X$ allows for the reconstruction of the foliation $\mathcal F(v)$.
\end{corollary}

\begin{proof} Consider a geodesic curve $[\g(g)]$ that connects two points, $x \in \d X$ and $C_v(x) \in \d X$. By Corollary \ref{foliations}, the segment $[\g]$ of the $v$-trajectory that connects $x$ and $C_v(x)$ is a geodesic curve. By the uniqueness hypotheses, we get $[\g(g)] = [\g]$. So the entire $\g$ can be reconstructed from $C_v$ and the $v$-harmonizing metric $g$.
\end{proof}




\begin{lemma}\label{lemA} Let $v$ be a traversing boundary generic vector field on a compact connected smooth Riemannian $(n+1)$-manifold $(X, g)$ with boundary. 
Then 
the measure $\mu_\Theta$ on $\d X$ can be recovered from the $g$-induced volume $n$-form $\Omega^\d$ on $\d X$ and the function $$\phi^\d =_{\mathsf{def}}(\Theta |_{\d X})/ (\Omega^\d):\; \d X \to \R$$ 
---the $\cos$ of the angle, formed by $v$ and the inner normal $\nu_g$ to $\d X$.\smallskip
\end{lemma}

\begin{proof} 
By the definition of the auxiliary function $\phi^\d: \d X \to [-1, 1]$ and using that $\Theta |_{K(df)}$ and $\Omega^\d$ on $\d X$ both are the $g$-induced volume $n$-forms, we have $\Theta(w) = \phi^\d(x) \cdot \Omega^\d(w)$ for any $w \in \Lambda^n T_x(\d X)$.
\hfill 
\end{proof}


By collapsing each $v$-trajectory to a point, we get a quotient {\sf trajectory space} $\mathcal T(v)$. We denote by $\Gamma: X \to \mathcal T(v)$ the quotient map, and by $\Gamma^\d: \d X \to \mathcal T(v)$ its restriction to $\d X$.\smallskip

For a traversing boundary generic vector field $v$, let $\mathcal Y$ denotes the set of $v$-trajectories that pass through the tangency locus $\d_2X(v)$. Let $\mathcal K(v) =_{\mathsf{def}} \mathcal Y \cap \d_1^+X(v)$. For such a $v$, the set $\mathcal K(v)$ is compact $(n-1)$-dimensional $CW$-complex, so  $\mu_\Theta(\mathcal K(v)) = 0$. In fact, $\Gamma^\d : \d_1^+X(v) \to \mathcal T(v)$ is a homeomorphism on the complement to the zero-measure set $\mathcal K(v)$.\smallskip

This observation motivates the following definition.

\begin{definition}\label{def2.4} 
Let $v$ be a boundary generic traversing vector field on a compact connected smooth $(n+1)$-manifold $X$ with boundary, and let $\Theta$ be a differential $n$-form that is integrally dual of $v$ (as in Lemma \ref{lem2.1}).  
\begin{itemize}
\item We introduce a measure $\tilde\mu_\Theta$ on the trajectory space $\mathcal T(v)$ by the formula $$\tilde\mu_\Theta(A) =_{\mathsf{def}} \mu_\Theta((\Gamma^\d)^{-1}(A)),$$ where $A \subset \mathcal T(v)$ is such that its $\Gamma^\d$-preimage is Lebesgue-measurable in the compact manifold $\d_1^+X(v)$.\smallskip

\item Then we interpret the integral $\int_{\d_1^+X(v)} \Theta$ as the \textsf{volume of the trajectory space} $\mathcal T(v)$ with respect to the measure $\tilde\mu_\Theta$ on $\mathcal T(v)$, induced by $\Theta$. 
\hfill $\diamondsuit$
\end{itemize}
\end{definition}

\begin{theorem}\label{iso_for_X} Let $v$ be a smooth traversing vector field on a smooth compact connected manifold with boundary. For any $v$-harmonizing pair $(g, df)$, the $g$-induced volume form $\Omega^\d$ on $\d X$, and the $\cos$-function $\phi^\d: \d X \to [-1, 1]$ allow for a computation of the $\Theta$-induced volume 
of the trajectory space $\mathcal T(v)$ via each of the following two formulas:  
$$vol_\Theta(\mathcal T(v)) = \pm \int_{\d^\pm X(v)} \phi^\d \cdot \Omega^\d.$$
Therefore, letting $g^\d := g|_{\d X}$, we get $vol_\Theta(\mathcal T(v))\; \leq \; vol_{g^\d}(\d^\pm X(v))$, which implies that  
\begin{eqnarray}\label{volume}
vol_\Theta(\mathcal T(v))\; \leq \;\frac{1}{2} vol_{g^\d}(\d X(v)).
\end{eqnarray}
\smallskip

Assuming that the restriction $f^\d: \d X \to \R$ of the Lyapunov function $f: X \to \R$ takes values in the interval $[0, 1]$,\footnote{which is always possible to achieve by an affine transformation of the target $\R$} for any such $v$-harmonizing pair $(g, df)$,
we get the ``holographic" isoperimetric inequality
\begin{eqnarray}\label{perimetric}
vol_g(X)\; \leq \; vol_{g^\d}(\d X).
\end{eqnarray}
\end{theorem}

\begin{proof} Put $\Theta = \ast_g(df)$. Examining Definition \ref{def2.4}, for each $v$-harmonizing pair $(g, df)$, the measure $\tilde\mu_\Theta$ on $\mathcal T(v)$  can be reconstructed from the following data:
\begin{itemize} 

\item the locus $\d_1^+X(v)$, 

\item the map $\Gamma^\d: \d_1^+X(v) \to \mathcal T(v)$ (whose \emph{generic} fiber is a singleton),

\item  the volume $n$-form $\Omega^\d$ on $\d_1^+X(v)$, induced by $g^\d$,

 \item  the ``$\cos$" function $\phi^\d$. 
\end{itemize}
\smallskip

Thus, we conclude that $vol_\Theta(\mathcal T(v)) = \int_{\d_1^+X(v)} \phi^\d \cdot \Omega^\d$, the volume of the trajectory space, can be reconstructed from the data in the first, third, and fourth bullet.
\smallskip

Since $df(v) > 0$ in $X$, we notice that $f$ attends its extrema on $\d X$. 
Hence, if $f: \d X \to [0, 1]$, then $f: X \to [0, 1]$.
Therefore, by Stokes' Theorem, we get 
$$vol_g(X) = \int_X df \wedge \Theta = \int_{\d X} f \cdot \Theta \leq \int_{\d X} 1\cdot  |\phi^\d|\cdot \Omega^\d \leq vol_{g^\d}(\d X).$$
Therefore, the volume of the bulk $X$ does not exceed the surface area of its boundary. This fact may please our fellow physicists who contemplate about black holes...
\end{proof}

\begin{definition}\label{well-balanced} Let $X$ be a compact connected smooth manifold with boundary, and $v$ a traversing vector field on it. A Lyapunov function $f: X \to \R$ is {\sf well-balanced} if $df(v) = 1$.

\hfill $\diamondsuit$
\end{definition}

By Theorem \ref{th2.1}, any traversing vector field admits a well-balanced Lyapunov function.\smallskip

Let $\mathsf{Diff}(X, \d X)$ be the group of the smooth diffeomorphisms of $X$ that are identities on $\d X$ and whose differentials are the identities on the bundle $TX|_{\d X}$. The group $\mathsf{Diff}(X, \d X)$ acts naturally on the space $\mathcal R(X)$ of smooth Riemannian metrics on $X$.

\begin{theorem}\label{th2.2} Let $v$ be a boundary generic traversing vector field on a compact connected manifold $X$ with boundary. Consider a  $v$-harmonizing pair $(g, df)$ (as in Definition \ref{harmonizing pair}), where the metric $g$ is $v$-\emph{invariant}, and the Lyapunov function is \emph{well-balanced}.

Assume that each $v$-trajectory $\g$ is either transversal to $\d X$ at \emph{some} point, or is quadratically tangent to $\d X$ at \emph{some} point $x$ so that $x = \g \cap \d X$.\footnote{This is the case when the $v$-flow is strictly concave or convex with respect to each component of $\d X$. We conjecture that this hypotheses is superfluous.}  
\smallskip

Then the following boundary-confined data: 
\begin{itemize}
\item the causality map $C_v: \d_1^+X(v) \to \d_1^-X(v)$, 
\item the restriction $g^\d = g|_{\d X}$ of the metric $g$ to the boundary, 
\item the restriction $f|_{\d X}$ of the Lyapunov function $f$ to the boundary,  
\item the angle-function $\theta: \d X \to S^1$, generated by $v$ and the inner normal  vector field $\nu$ to $\d X$ in $X$,
\end{itemize}
allow for a reconstruction of the smooth topological type of $X$ and of the metric $g$ on it, up to the natural $\mathsf{Diff}(X, \d X)$-action on the space $\mathcal R(X)$ of Riemannian metrics on $X$.
\end{theorem}

\begin{proof} 
By Theorem \ref{th2.1}, there is a $v$-harmonizing pair $(g, df)$ with a $v$-invariant $g$.

 Let $\mathcal M(f)$ be the codimension one foliation on $X$, defined by the connected components of the hypersurfaces of $f$-constant level. Recall that, due to the critical points of $f: \d X \to \R$, the leaves of $\mathcal M(f)$ may be singular. However, as before, we view $\mathcal M(f)$ as the intersection of a nonsingular foliation $\mathcal M(\hat f)$ on an open manifold $\hat X \supset X$ with $X$. 

Since $df(v) > 0$, every leaf of $\mathcal M(f)$ intersects with every leaf of $\mathcal F(v)$ at a singleton at most. Again, since $df(v) = 1$, $f$ attends its extrema on the boundary. As a result, any hypersurface  $f^{-1}(c)$ has a nonempty intersection with $\d X$. Moreover, each point $x \in X$ is uniquely determined by a point $y \in \d X \cap \g_x$, where $\g_x$ stands for the $v$-trajectory through $x$, and by the value $f(x)$. Thus, the pair of smooth foliations $(\mathcal M(\hat f),\, \mathcal F(\hat v))$ delivers a ``coordinate grid" for $X$, so that the intersections $\mathcal M(\hat f) \cap \d X$ and $\mathcal F(\hat v)  \cap \d X$ provide the ``holographic structure" from which $X$ will be recovered. We notice that  the \emph{ordered} finite set $\g_x \cap \d X$ may be interpreted as the $C_v$-trajectory of $y$, where $y \in \g_x \cap \d X$ is the minimal element.

By Definition \ref{harmonizing pair}, the distribution $K(df) \subset TX$ by the kernels of $df$ is the $g$-orthogonal compliment $K_v^\perp$ to the field $v$. Also $v \in K(\ast_g(df)) \subset TX$, the distribution by the kernels of $\ast_g(df)$. The leaves of $\mathcal M(f)$ and $\mathcal F(v)$ are $g$-orthogonal. We denote by $g^\perp$ the restriction of $g$ to the $n$-dimensional distribution $K^\perp_v$, and by $g^\uparrow$ the restriction of $g$ to the $1$-dimensional distribution $K_v := K(\ast_g(df))$. Since the pair $(df, g)$ is $v$-invariant, so are the pairs $(K^\perp_v, g^\perp)$ and $(K_v, g^\uparrow)$. Therefore, knowing the $v$-invariant restrictions $g^\perp$ and $g^\uparrow$ is sufficient for determining the metric $g = g^\perp \oplus g^\uparrow$.

On the other hand, by the $v$-invariant property of $g$, if we know $g|_{f^{-1}(c)}$ in the vicinity of a $v$-trajectory $\g$ for one particular value of $c \in \R$, then we know all the restrictions $\{g|_{f^{-1}(c')}\}_{c' \in \R}$ in vicinity of $\g$, provided $f^{-1}(c') \cap \g \neq \emptyset,\; f^{-1}(c) \cap \g \neq \emptyset$. 
Similarly, by the $v$-invariance of $(df, g)$, if we know the restriction of $g$ to the $\g$-tangent line at one particular point, we know the restriction of $g$ to the $\g$-tangent line at any other point along $\g$.

By the Holography Theorem \ref{th1.1}, the map $C_v$ determines the pair $(X, \mathcal F(v))$, up to a diffeomorphism $\Phi: X \to X$ that is the identity on $\d X$. The property of $\Phi$ being a diffeomorphism (and not just a homeomorphism) depends on the property of each $v$-trajectory $\g$ being either transversal to $\d X$ at some point, or being quadratically tangent to $\d X$  at some point (so that $\g \cap \d X$ is a singleton). By  the proof of Theorem \ref{th1.1}, the map $C_v$ determines the triple $(X, \mathcal F(v), \mathcal M(f))$, up to a diffeomorphism $\Phi: X \to X$ that is the identity on $\d X$ (see \cite{K4}, Theorem 4.1). 

 
Let $ (g_\tau)^\d$ denote the restriction of the metric $g$ to $T_\ast(\d X)$, and $(g_\nu)^\d$ to the normal bundle $\nu(\d X, X)$.

Since $\Phi$ is assumed to fix the boundary $\d X$ and the map $\theta: \d X \to S^1$, where $\theta := \angle_g(v, \nu)$, its action on the bundle $TX|_{\d X}$ is trivial.

Let $(g^\uparrow)^\d$ denotes the restriction of the metric $g^\uparrow$ to the foliation $\mathcal F(v) |_{\d X}$, and let $(g^\perp)^\d$ denotes the restriction of  $g^\perp$ to the foliation $\mathcal M(f) |_{\d X}$. The knowledge of $g^\d$ and $\theta$ makes it possible to determine the orthogonal decomposition $g|_{\d X} =  (g^\uparrow)^\d \oplus   (g^\perp)^\d$ along $\d X$.  Note that the plane, spanned by the vectors $\nu(x)$ and $v(x)$ at $x \in \d X$, is orthogonal to the subspace $T_x(\d X) \cap K^\perp_{v(x)}$. The orthogonal $(2 \times 2)$-matrix $A(\theta)$, representing the rotation on the angle $\theta$, connects the decomposition $g|_{\d X} = (g^\uparrow)^\d \oplus  (g^\perp)^\d$ to the decomposition $g|_{\d X}  =  (g_\nu)^\d\, \oplus \,(g_\tau)^\d$. So knowing $(g_\nu)^\d \oplus (g_\tau)^\d$ and $\theta$ determines $(g^\uparrow)^\d \oplus (g^\perp)^\d$.  


By the $v$-invariant property of  $(g, df)$,  the decomposition $g|_{\d X} = (g^\uparrow)^\d \oplus  (g^\perp)^\d$  spreads uniquely to an orthogonal decomposition $g =  g^\uparrow\, \oplus \, g^\perp$ in $X$.  

Therefore the quadruple $(f|_{\d X},\,  g|_{\d X},\, \theta,\, C_v)$ determines $g$, up to the natural $\Phi$-action. 
\hfill
\end{proof}
\section{The scattering maps, the $v^g$-harmonizing metrics, and their isoperimetric inequalities}

The lemma below describes conditions under which the scattering map $C_{v^g}$ (see Fig.2) is a  $\mu_\Theta$-\emph{measure-preserving} transformation for the appropriate choice of $(2n-2)$-form $\Theta$ on $SM$. In a way, this lemma it is a special case of Lemma \ref{lem2.1}.

\begin{lemma}\label{lem5.2} Assume that a metric $g$ on a compact connected Riemannian $n$-manifold $M$ with boundary is non-trapping and boundary generic. 

Pick a  $(2n-2)$-volume form $\Omega^\d$ on $\d(SM)$, consistent with its orientation. 
Let $\Theta$ be a differential $(2n - 2)$-form on $SM$  such that:
\begin{itemize}
\item $\dim(\ker(\Theta)) =1$,
\item the geodesic field $v^{g} \in \ker(\Theta)$,
\item $d \Theta = 0$,
\item the function $\Theta/\Omega^\d: \d(SM) \to \R$ is positive (negative) only on the interior of $\d_1^+(SM)$ (of $\d_1^-(SM)$). 
\end{itemize}

Then the scattering map $C_{v^g}: \d_1^+(SM) \to \d_1^-(SM)$ preserves the $\Theta$-induced measure $\mu_\Theta$ on $ \d_1(SM)$, i.e.,
$$\Big |\int_{C_{v^g}(K)}  \Theta \Big | =  \int_K \Theta $$
for any Lebesgue-measurable set $K \subset  \d_1^+(SM)$. 
\end{lemma}

\begin{proof} Since the metric $g$ on $M$ is non-trapping, the geodesic field $v^g$ is  traversing on $SM$ (\cite{K5}, Lemma 2.2). Since the metric $g$ is boundary generic (see Definition 2.4 and Lemma 3.3 in \cite{K5}) relative to $\d M$, the field $v^g$ is boundary generic relative to $\d(SM)$. 
Thus, Lemma \ref{lem2.1} is applicable to $v^g$. By that lemma, we get that $C_{v^g}$ preserves the measure $\mu_\Theta$.  
\hfill 
\end{proof}

%
\begin{figure}[ht]\label{figAA}
\centerline{\includegraphics[height=1.7in,width=3in]{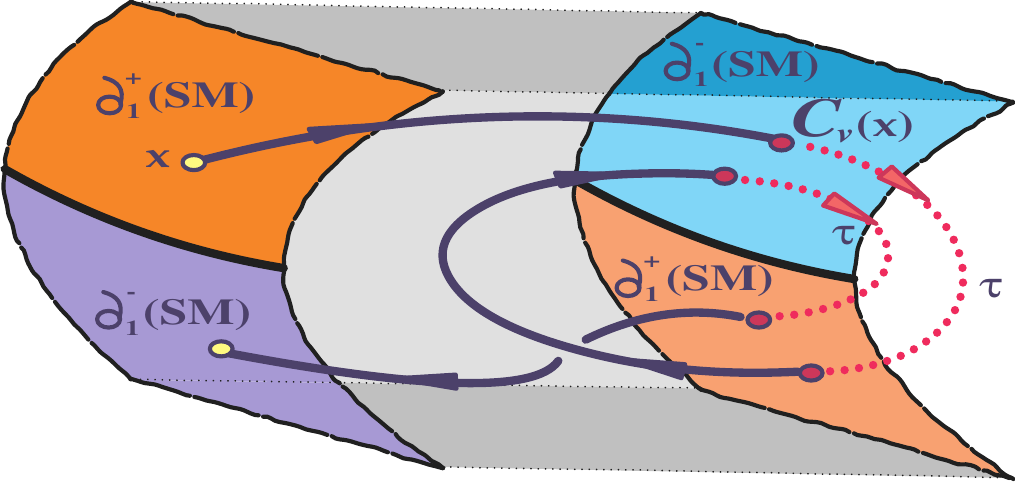}}
\bigskip
\caption{\small{A scattering (causality) map $C_{v^g}: \d_1^+(SM) \to \d_1^-(SM)$ and an involution $\tau: \d_1^-(SM) \to \d_1^+(SM)$, mimicking the elastic reflection from $\d M$. The two bold arcs that separate $\d_1^+(SM)$ and $\d_1^-(SM)$ represent the locus $\d_2(SM)$.}} 
\end{figure}

\begin{lemma}\label{how to get well-balanced} Let $M$ be a codimention zero smooth compact submanifold of a compact connected Riemannian manifold $(L, g)$ so that $M \cap \d L = \emptyset$.  Assume that each geodesic in $L$ that intersects with $M$ hits $\d L$ transversally at a pair of distinct points\footnote{This condition means that $\d L$ looks ``convex", as seen from $M$.}. 

Then the metric $g|_M$ is non-trapping, and the geodesic flow on $SM$ admits a well-balanced Lyapunov function.
\end{lemma}

\begin{proof} Each geodesics $\g$ in $M$ extends to a unique geodesic $\g_L$ in $L$, which intersects with $\d L$ at a pair of points $a(\g_L)$ and $b(\g_L)$. The orientation of $\g$ picks one of these two points, say $a(\g_L)$, as the starting point of $\g_L$. 
Consider the function $f(x) := \ell_g(x, a(\g_L))$, where $x \in \g_L$ and $\ell_g(x, a(\g_L))$ is the length of the geodesic arc $[x, a(\g_L)] \subset \g_L$. Let $\tilde\g_L$ be a lift of $\g_L$ to the spherical bundle $SL$. Topologically, $\tilde\g_L$ is a closed interval. 

For any point $(m, v) \in SM$, we take the trajectory $\tilde\g_L((m, v))$ of the geodesic flow on $SL$ through $(m, v)$. We define the function $F: \tilde\g_L((m, v)) \to \R$ as the pull-back of $f: \g_L \to \R$. In fact, $F((n, v))$ is the length (in the Sasaki metric $gg$) of the segment $[(a(\g_L), w(\g_L)),\, (n, v)]$ of $\tilde\g_L$, where $w(\g_L)$ is the tangent to $\g_L$ vector at $a(\g_L)$ and $v$ is tangent to $\g_L$ at $n \in L$. Using that $\tilde\g_L((m, v))$ is transversal to $\d(SL)$ at the point $(a(\g_L), w(\g_L))$, we get that $F$ is a smooth function of $(n, v)$.  Indeed, the solutions ODEs depend smoothly on the initial data, provided that such data vary along a transversal section of the flow. Evidently, $dF(v^g) > 0$ since $f$ is increasing along $\g_L$. The variation of $F$ along any segment $\D$ of $\tilde\g_L$ equals to its length in $gg$, which, in turn, equals the length of $\pi(\D)$ in $g$. In particular, for any $\D \subset \tilde\g_L \cap SM$, the variation of $F$ along $\D$ is the length of $\pi(\D)$ in $g$. Thus $F$ is well-balanced in $M$. 
\end{proof}

\begin{corollary}\label{hyperbolic} 
Let $(N, g^\dagger)$ be a closed hyperbolic (or flat) $n$-dimensional manifold, and $M$ a compact smooth $n$-manifold, equipped with a submersion $p: M \to N$ such that the natural homomorphism $p_\ast: \pi_1(M) \to \pi_1(N)$ of the fundamental groups is trivial.
Let $g$ denote the $p$-induced pull-back to $M$ of the hyperbolic (or flat) metric $g^\dagger$ on $N$.

Then $g$ is a non-trapping metric, and $(M, g)$ admits a well-balanced Lyapunov function  $F: SM \to \R$. \smallskip

In particular, any compact smooth domain $M$ in the hyperbolic space $\H^n$ or Euclidean space $\E^n$  admits a non-trapping hyperbolic or Euclidean metric with a well-balanced Lyapunov function  $F: SM \to \R$. And so does any sufficiently small domain $M$ in the spherical space $\S^n$. 
\end{corollary}

\begin{proof} Since the homomorphism $p_\ast : \pi_1(M) \to \pi_1(N)$ is trivial, $M$ admits a 
lifting $\tilde p:  M \to \tilde N$ to the universal cover $\tilde N = \H^n$ (or $\tilde N = \E^n$) of $N$. This lifting is also a submersion. Hence, it suffices to treat the case of a submersion $\tilde p: M \to \H^n$ (or of $\tilde p: M \to \E^n$) which pulls back the hyperbolic (or Euclidean) metric to form $g$. 

Let us consider the hyperbolic case first. 
In the Poicar\'{e} model, $\H^n$ is identified with the interior of the unit Euclidean ball $B^n$, and the geodesics in $\H^n$ with the arcs of circles in $B^n$ (or diameters through the origin)  that are orthogonal to its boundary $\d B^n$. 

Let $B^n(r) \subset B^n$ be a concentric Euclidean ball of radius $r \in (0, 1)$. Since $M$ is compact, we may assume that $\tilde p(M) \subset B^n(r_0)$ for some $r_0 < 1$. Then there is $r_\star \in (r_0, 1)$ such that any geodesic arc $\hat\g$ that intersects with $B^n(r_0)$ hits $\d B^n(r_\star)$ at a pair of points and is transversal to $\d B^n(r_\star)$ at the intersections. Let us explain informally the last claim (we leave the details to the reader): in the hyperbolic metric, as $r_\star \to 1$, $B^n(r_0)$ becomes infinitesimally small relative to $B^n(r_\star)$, and the geodesic arcs through $B^n(r_0)$ approach the diameters of $B^n$; at the same time, in the Euclidean metric, $\d B^n(r_\star)$ approaches $\d B^n$.  

Now, picking $r_\star$ very close to $1$ and the ball $B^n(r_\star)$ for the role of $L$ in Lemma \ref{how to get well-balanced}, we validate the corollary (in the hyperbolic case) by pulling back, with the help of the bundle map $S\tilde p: SM \to S\H^n$,  the Lyapunov function on $SB^n(r_0)$, built in the lemma. \smallskip

The Euclidean case is similar: we pick an Euclidean ball $B^n(r_0) \supset \tilde p(M)$, where the radius $r_0 \in (0, +\infty)$. Then there is $r_\star \in (r_0, +\infty)$ such that any line $\hat\g$ that intersects with $B^n(r_0)$ hits the sphere $\d B^n(r_\star)$ at a pair of points and is transversal there to $\d B^n(r_\star)$.
Specifically, for a pair $(\vec a, \vec v)$, where $\vec a \in \E^n$ is such that $\|\vec a\| \leq r_0 < r_\star$ and $\vec v \in T_{\vec a}(\E^n)$ has the norm $1$, we may pick the composition of the submersion $S\tilde p$ with the Lyapunov function
\begin{eqnarray}\label{eq.Lyapunov_for_E}
F(\vec a, \vec v) := - \big(r_\star^2 - \|\vec a\|^2 + \langle \vec a, \vec v \rangle^2 \big)^{\frac{1}{2}} + \langle \vec a, \vec v \rangle.
\end{eqnarray} 
This composition serves as a Lyapunov function on $SM$ for any submersion $p$ such that $\tilde p(M) \subset B^n(r_0)$, where $r_0 < r_\star$.

Let a finite group $G$ act freely and smoothly on the sphere $\S^n$. Finally, the spherical case $p: M \to \S^n/G$ requires $\tilde p(M) \subset \S^n$ being so small that any geodesic circle that intersects with $\tilde p(M)$ hits transversally some fixed equator $\S^{n-1} \subset \S^n$. 
\end{proof}


The next theorem is a direct application of Theorem \ref{th2.1}, claim (3), and  Theorem \ref{traversing_harmonizing}, together with Corollary \ref{foliations},  to the geodesic flows on Riemannian manifolds with boundary that admit a non-trapping metric $g$.


\begin{theorem}\label{foliations on SM} Let $g$ be a boundary generic non-trapping Riemannian metric on a smooth compact connected $n$-manifold $M$ with boundary. Then the following claims hold:
\begin{itemize}
\item The geodesic vector field $v^g$ admits a well-balanced Lyapunov function $F: SM \to \R$ 
and a $v^g$-harmonizing pair $(g^\bullet, dF)$, where $g^\bullet$ is a $v^g$-invariant Riemmanian metric on $SM$.  \smallskip

\item In the metric $g^\bullet$, the leaves of the $1$-foliation $\mathcal F(v^g)$ are geodesic curves\footnote{As they are also in the Sasaki metric $gg$.}, and the leaves $\{F = c\}_{c \in \R}$ of the orthogonal $(2n-2)$-foliation $\mathcal G(F)$ are the volume-minimizing relative hypersurfaces in $\big(SM,\, \d(SM)\big)$. \smallskip

\item The $(2n-2)$-form $\Theta =_{\mathsf{def}}  \ast_{g^\bullet}(d F)$ has all the  properties from Lemma \ref{lem5.2}. As a result, such a form $\Theta$ defines a measure $\mu_{\Theta}$ on $\d(SM)$, which is preserved by the scattering map $C_{v^g}: \d_1^+(SM) \to \d_1^-(SM)$. 
\hfill $\diamondsuit$
\end{itemize}
\end{theorem}

\smallskip

\begin{theorem}\label{th3.2}
Let $g$ be a boundary generic non-trapping Riemannian metric on a smooth compact connected $n$-manifold $M$ with boundary, and let $F: SM \to \R$ be a Lyapunov function for the geodesic vector field $v^g$.

Let $\Theta$ be a $(2n-2)$-form on $SM$ with the properties as in Lemma \ref{lem5.2}. In particular, we may choose a nil-divergent $\Theta := v^g \rfloor \Omega$, where $\Omega$ is any $v^g$-invariant volume form on $SM$. 
\smallskip

Then the following claims hold:
\begin{itemize}
\item There exists a Riemannian metric $g^\bullet$ on $SM$ such that $\ast_{g^\bullet}(dF) = \Theta$.

\item The function $F$ may be chosen to be well-balanced and the metric $g^\bullet$ to be  $v^g$-invariant.

\item In the $\Theta$-induced measure, the volume of the space $\mathcal T(v^g)$ of geodesics on $M$ satisfies the inequality: 
$$vol_{\Theta}(\mathcal T(v^g)) =_{\mathsf{def}}\; \int_{\d_1^\pm(SM)} \Theta \; \leq \;  vol_{(g^\bullet)^\d}(\d_1^\pm(SM)),$$
where $(g^\bullet)^\d$ denotes the restriction of the metric $g^\bullet$ to the boundary $\d(SM)$.
As a result, 
\begin{eqnarray}
vol_{\Theta}\big(\mathcal T(v^g)\big) \,  \leq \; \frac{1}{2}\, vol_{(g^\bullet)^\d}\big(\d(SM)\big).
\end{eqnarray}

\item Assuming that the Lyapunov function $F: SM \to \R$ is chosen so that  $F(\d(SM)) \subset [0, 1]$,  the following isoperimetric inequality holds:
\begin{eqnarray}
vol_{g^\bullet}(SM) \,  \leq \;  vol_{(g^\bullet)^\d}(\d(SM)).
\end{eqnarray}
\end{itemize}
\end{theorem}

\begin{proof} By the properties of the non-trapping metric $g$ on $M$, the geodesic field $v^g$ on $SM$ is traversing (\cite{K5}, Lemma 2.2).  So there exists a smooth function $F: SM \to \R$ with the property $dF(v^g) > 0$ (\cite{K1}, Lemma 4.1). By Theorem \ref{th2.1}, $F$ may be chosen to be well-balanced. Also by Theorem \ref{th2.1}, there exists a $v^g$-harmonizing and $v^g$-invariant metric $g^\bullet$ on $SM$ (in which $dF$ is co-closed). 
In particular, along the locus $\d_1^+(SM)$, the field $v^g$ points inside of $SM$, so that the closed form $\Theta := \ast_{g^\bullet}(d F)|_{\d_1^+(SM)}$ is positively proportional to the $g^\bullet$-induced volume form in the interior of $\d_1^+(SM)$. The coefficient of proportionality is $\cos\big(\angle(\nu, v^g)\big)$, where $\nu$ is the inward normal (in the metric $g^\bullet$) to $\d(SM)$ in $SM$. 

By Corollary  \ref{foliations}, the leaves of the $1$-foliation $\mathcal F(v^g)$ are geodesic curves in the $v^g$-harmonizing metric $g^\bullet$, and the leaves of the orthogonal $(2n-2)$-foliation $\mathcal G(F) =_{\mathsf{def}} \{F^{-1}(c)\}_{c \in \R}$ are the volume-minimizing hypersurfaces $H$ (among all hypersurfaces $\Sigma$ with the property $\Sigma \cap \d(SM) = H \cap \d(SM)$). This validates the first two bullets of the theorem. 

The last two bullets follow directly from Corollary \ref{iso_for_X} and Theorem \ref{th2.1}. \hfill
\end{proof}

\noindent {\bf Remark 5.1} Note that $vol_{\Theta}(\mathcal T(v^g)) = vol_{\Theta}(\mathcal T(-v^g))$, so that the volume of $\mathcal T(v^g)$ can be also expressed in terms of integration over the locus $\d_1^+(SM)(-v^g) = \d_1^-(SM)(v^g)$. \hfill $\diamondsuit$

\begin{theorem}\label{th3.3} 
Let $g$ be a boundary generic non-trapping Riemannian metric on a smooth compact connected $n$-manifold $M$ with boundary. Assume that any geodesic curve in $M$ is either transversal to $\d M$ or is simply tangent to $\d M$ at \emph{some} point. 

We choose a $v^g$-invariant and $v^g$-harmonising pair $(g^\bullet, dF)$ on $SM$ with the Lyapunov function $F$ being well-balanced \footnote{By Theorem \ref{foliations on SM}, such a pair $(g^\bullet, dF)$ exists.}.\smallskip

Then the scattering map $C_{v^g}: \d_1^+(SM) \to \d_1^-(SM)$, the restriction $F^\d: \d(SM) \to \R$ of $F$ and of the metric $g^\bullet$ to the boundary $\d(SM)$, and the $g^\bullet$-induced angle map $\theta: \d(SM) \to S^1$ allow for a reconstruction of: 
\begin{itemize}
\item the space $SM$, 

\item the geodesic vector field $v^g$, 

\item and of the metric $g^\bullet$ on $SM$, 
\end{itemize}
up to the natural action of diffeomorphisms 
$\Phi: SM \to SM$ that are the identity map on $\d(SM)$ and whose differentials are the identity map on the bundle $T(SM)|_{\d(SM)}$. 

\end{theorem}

\begin{proof} By Theorem \ref{traversing_harmonizing} and Theorem \ref{th2.2}, the knowledge of  $F^\d: \d(SM) \to \R$, the metric $(g^\bullet)^\d$ on $SM$, and angle map $\theta: \d(SM) \to S^1$, allow for a reconstruction of the smooth topological type of $SM$, the foliation $\mathcal F(v^g)$, and the metric $g^\bullet$, provided that $g^\bullet$ is $v^g$-invariant. Using that $dF(v^g) = 1$, we can reconstruct $v^g$ from $\mathcal F(v^g)$. In general, by Theorem \ref{th1.2}, the reconstruction is possible up to a homeomorphism $\Phi: SM \to SM$ that is fixed on $\d(SM)$ and whose restriction to each $v^g$-trajectory is an orientation-preserving diffeomorphism. Under the hypotheses of Theorem \ref{th3.3}, again by Theorem \ref{th1.2}, $\Phi$ is a smooth diffeomorphism. 
With the help of the $\Phi$-invariant map $\theta: \d(SM) \to S^1$, we conclude that the differential $D\Phi$ must act trivially on $T(SM)|_{\d(SM)}$.

However, assuming that the smooth topological type of $SM$ is known, we may drop the assumption that any geodesic curve in $M$ is either transversal to $\d M$ at some point or is simply tangent to $\d M$: indeed, the knowledge of the metric $g^\bullet$ on the complement to the locus $SM(v^g, (33)_\succeq \cup (4)_\succeq)$ of codimension $3$ at least, by continuity, is sufficient for the reconstruction of $g^\bullet$ everywhere.
\end{proof}

Moving away from the non-trapping metrics, we propose the following notions.

\begin{definition}\label{geodesically harmonic} Let $N$ be a smooth compact manifold and let $g^\bullet$ be a  Riemmanian metric on $SN$ such that $\|v^g\|_{g^\bullet} =1$. Consider a  $1$-form $\a_\bullet = \a_\bullet(g^\bullet)$ on $SN$ that is defined by the two properties: $(1)$\, 
 $\ker(\a_\bullet) \perp_{g^\bullet} v^g$, and $(2)$\,  $\a_\bullet(v^g) = 1$.

We say that a Riemmanian metric $g$ on $N$ is \textsf{geodesically harmonic}, if there exists a Riemmanian metric $g^\bullet$ on $SN$ as above and such that the $1$-form $\a_\bullet$ is closed and co-closed.

A Riemmanian metric $g$ on $N$ is \textsf{invariantly geodesically harmonic}, if it is geodesically harmonic and the metric $g^\bullet$ is $v^g$-invariant.
\smallskip

We denote by $\mathcal{H}ar(N)$ ($\mathcal{H}ar_{inv}(N)$) the space of (invariantly) geodesically harmonic metrics on $N$. \hfill $\diamondsuit$
\end{definition}

By Theorem \ref{foliations on SM}, any non-trapping and boundary generic metric $g$ on a connected compact $M$ is invariantly geodesically harmonic. Indeed, just follow the well-traveled path (see the proof of Theorem \ref{th2.1}): take a well-balanced Lyapunov function $F$, put $\a = dF$, construct a closed $v^g$-invariant $n$-form $\Theta$, whose kernel is spanned by $v^g$ and such that $\a \wedge \Theta > 0$, and finally, construct a $v^g$-invariant metric $g^\bullet$ so that $\Theta = \ast_{g^\bullet}(\a)$. Then $\ker(\Theta) \perp_{g^\bullet} \ker(dF)$. Thus, $g$ is invariantly geodesically harmonic in the sense of Definition \ref{geodesically harmonic}.\smallskip

However, the geodesically harmonic metrics on \emph{closed} manifolds $M$ seem to be a rare phenomenon. For instance, if $H^1(SM; \R)=0$ (for $\dim(M) \geq 3$, this is equivalent to $H^1(M; \R)=0$) no such metric $g^\bullet$ is available, since the cohomology class of $\a$ is nontrivial.\smallskip

Recall also that the $v^g$-\emph{invariant} Sasaki's metrics $gg$ on $SM$ are extremely rare indeed. According to \cite{Be}, Proposition 1.104, $gg$ is $v^g$-invariant if and only if the sectional curvature of $(M, g)$ is identically $1$. 
\smallskip

\noindent {\bf Remark 5.1.} Note that if $(N, g)$ is (invariantly) geodesically harmonic, then any smooth codimension zero proper submanifold $(M, g|_M) \subset (N, g)$ (with boundary generic metric $g|_M$) is automatically (invariantly) geodesically harmonic. \hfill $\diamondsuit$
\smallskip


Let us describe numerically how far a given metric $g^\bullet$ on $SM$ is from being ``harmonic".  Speaking informally, we would like to measure the ``distance" between the standard {\sf contact structure} $\b_g$ on $SM$ (see (\ref{eq.beta})) and {\sf taught foliations} $\mathcal G$ on $SM$ that are transversal to $v^g$. 

Let $\|\sim \|_2^{g^\bullet}$ denote the $L_2$-norm of differential forms on $SM$ in the metric $g^\bullet$. Let the $1$-form $\a_\bullet$ be as in Definition \ref{geodesically harmonic}. 
Then the quantity  $$\delta(g^\bullet) : = \sqrt{\big(\|d\a_\bullet \|_2^{g^\bullet}\big)^2 + \big(\|d(\ast_{g^\bullet}(\a_\bullet))\|_2^{g^\bullet}\big)^2}$$ measures the failure of a  candidate metric $g^\bullet$ to deliver the geodesic harmonicity of $g$. 

Evidently, if $g$ is geodesically harmonic, then $\delta(g^\bullet) = 0$ for an appropriate $g^\bullet$. \smallskip

\begin{definition} 
Let $(M, g)$ be a connected compact Riemannian manifold. Consider the number 
\begin{eqnarray}\label{D(g)}
\mathcal D(g) : = \inf_{\{g^\bullet\}}\, \big\{\delta(g^\bullet)\big\},
\end{eqnarray}
where $g^\bullet$ runs over all Riemannian metrics on $SM$ such that $vol_{g^\bullet}(SM) =1$. 
\smallskip

If $g^\bullet$ runs over all $v^g$-\emph{invariant} Riemannian metrics on $SM$ 
such that $vol_{g^\bullet}(SM) =1$, we get a similar to (\ref{D(g)}) quantity $\mathcal D_{inv}(g) \geq \mathcal D(g)$. \hfill $\diamondsuit$
\end{definition}

Recall that $\mathcal D_{inv}(g) = \mathcal D(g) = 0$ for any non-trapping $g$ on a compact connected manifold with boundary.

\begin{conjecture}\label{D(g) = 0} If $\mathcal D(g) = 0$, then $g$ is geodesically harmonic. If $\mathcal D_{inv}(g) = 0$, then $g$ is invariantly geodesically harmonic. \hfill $\diamondsuit$
\end{conjecture}

In the spirit of Conjecture \ref{D(g) = 0}, $\mathcal D(g)$ should measure the failure of geodesic harmonicity for a given metric $g$ on $M$.
\smallskip

\noindent {\bf Problem 5.1.} 
\begin{itemize}
\item Estimate $\delta(gg)$ for the Sasaki metric $gg$ on $SM$.
\smallskip

\item Compute $\delta(gg)$ for the $v^g$-invariant Sasaki metric $gg$ on any compact symmetric space $(N, g)$ of rank one. By \cite{Be}, these are Riemmanian manifolds $\S^n$, $\R \P^n$, $\C \P^n$, $\H \P^n$, and $\C a\P^2$ with constant sectional curvature $1$. \hfill $\diamondsuit$
\end{itemize}
\smallskip

\section{On the holography of billiard maps for non-trapping metrics}

In this section, we will derive direct applications of the results from the previous sections to the geodesic flows of non-trapping metrics on manifolds with boundary.\smallskip

 Let $(M, g)$ be a compact connected smooth Riemannian $n$-manifold with boundary and $v^g$ the geodesic vector field on the tangent unitary bundle $SM$. 
 
 Any point on the boundary of $SM$ is represented by a pair $(x, w)$, where $x \in \d M$ and $w \in T_xM$ is a unit vector. In the local coordinates $(\vec q, \vec p)$ on $TM$, the spherical fibration $SM \subset TM$ is the locus $\{H(\vec q, \vec p) =1\}$ for  the Hamiltonian 
 \begin{eqnarray}\label{H}
 H(\vec q, \vec p) := \frac{1}{2}\sum_{1\leq i,j \leq n} g_{ij}(q)\, p_ip_j.
 \end{eqnarray}
 
It is easy to see that the loci $\d_1^\pm(SM)(v^g)$ are $g$-independent (for example, $\d_1^+(SM)(v^g)$ consists of pairs $(x, w)$, where $x \in \d M$ and $w \in T_xM$ belongs to the closed half-space, whose vectors point inside of $M$). So we will denote these loci  by ``$\d_1^\pm(SM)$". 
\smallskip
 
 Any tangent vector $\dot q \in T_qM$, where $q \in \d M$, is a sum of $a\cdot n + b \cdot t$, where $n$ is the inner normal to $\d M$ in $M$ (with respect to the metric $g$), $t \in T_q(\d M)$, and $a, b \in \R$.  Consider a smooth involution $\tau_g: \d(SM) \to \d(SM)$ that takes any tangent to $M$ vector $\dot q  = a\cdot n + b \cdot t$ at $q \in \d M$, to the vector $\tau_g(\dot q) = -a\cdot n + b \cdot t$, the orthogonal reflection of $\dot q$ with respect to the hyperplane $T_q(\d M)$. 
The involution $\tau_g$ induces a $g$-isometry of each sphere $\{S_qM \subset T_qM\}_{q \in \d M}$ with respect to its equator. 
  Evidently, $\tau_g$ maps $\d_1^\pm(SM)$ to $\d_1^\mp(SM)$ (see Fig.2). 
 \smallskip
 
 One can generalize the construction of the mirror reflection $\tau_g: \d(SM) \to \d(SM)$ in the spirit of Finsler structures \cite{Car} as follows. Consider any smooth involution $\tau: \d(SM) \to \d(SM)$ which is a map of the spherical fibration $\eta: \d(SM) = SM|_{\d M}  \to \d M$, which is the identity on its base $\d M$ and such that $\tau$-fixed locus $\d(SM)^\tau$ is $S(\d M) = \d(SM) \cap T(\d M)$. 
 
For example, we may consider a new smooth Riemannian metric $\tilde g$ in the vector bundle $TM|_{\d M} \to \d M$ and the $\tilde g$-generated spherical fibration $\tilde SM \to \d M$. Then, for each point $x \in \d(SM)$, we take the ray $\ell_x$ in $T_xM$ through $x$ and the origin, produce the unique point $y \in \tilde SM$ that belongs to $\ell_x$, apply the reflection $\tau_{\tilde g}$ to $y$, and finally produce the point $\tau(x) := \ell_{\tau_{\tilde g}(y)} \cap SM$. 
\smallskip
 
The results of this section rely on the existence of a differential $(2n-2)$-form $\Theta$ on $SM$, subject to the following properties (that mirror the list in Definition \ref{def.dual}):
\smallskip
\begin{eqnarray}\label{properties_of_Theta}
\end{eqnarray}
\begin{itemize}
\item $d \Theta = 0$,
\item $\dim(\ker(\Theta)) = 1$,
\item $v^g \in \ker(\Theta)$,
\item $\pm\Theta|_{\d_1^\pm(SM)} \geq 0$,
\item $\tau^\ast(\Theta |_{\d_1(SM)}) = \Theta|_{\d_1(SM)}$ with respect to a given involution $\tau: \d(SM) \to \d(SM)$ as above.
\end{itemize}
\smallskip

Let us recall now few basic constructions  from Symplectic Geometry. Let $\b$ be the tautological Liouville $1$-form  (locally, ``$\vec {p^\ast}\cdot \vec{dq}$") on  the cotangent bundle $T^\ast M$, viewed as a smooth section of the bundle $T^\ast(T^\ast M) \to T^\ast M$. Let $\om = - d \b$ be the canonic  symplectic form on $T^\ast M$ (locally, ``$\vec {dq}\wedge \vec{dp^\ast}$"). 

The metric $g$ on $M$ gives rise to a bundle isomorphism $\Phi_g: TM \to T^\ast M$. Consider the pull-backs $\b_g =_{\mathsf{def}} \Phi_g^\ast(\b)$ of $\b$ and 
$\om_g =_{\mathsf{def}} \Phi_g^\ast(\om)$ of $\om$ under the diffeomorphism $\Phi$. 

Let $\mu$ be the unitary (in $g$) radial vector field, normal to $SM$ in $TM$ and tangent to the fibers of $TM \to M$. 
Then $\b_g = \pm\, \mu\,  \rfloor \, \om_g$. 

In local coordinates $(q^1, \dots, q^n,\, p^1, \dots, p^n)$ on $TM$, these forms may be written as 
\begin{eqnarray}\label{eq.beta}
\b_g = \sum_{i, j} g_{ij}\,p^i\, dq^j,
\end{eqnarray}
and 
\begin{eqnarray}\label{eq.omega}
\om_g = \sum_{i, j} g_{ij}\, dq^i \wedge dp^j + \sum_{i, j, k} \frac{\partial g_{ij}}{\partial q^k}\, p^i\, dq^j \wedge dq^k.
\end{eqnarray}

By the Liouville theorem, both $\b_g$ and $\om_g = d\b_g$ are $v^g$-invariant forms, that is, $\mathcal L_{v^g} (\b_g) = 0$ and $\mathcal L_{v^g}(\om_g) = 0$.  Moreover, using the formulas (\ref{eq.beta}) and (\ref{H}), we get $\b_g(v^g) = 1$.
\smallskip

We denote by $(\om_g)^n$, the $n^{th}$ exterior power of the $2$-form $\om_g$. It is a $(2n)$-dimensional volume form on $TM$. Let $\Omega_g =_{\mathsf{def}} \{\pm \b_g \wedge (\om_g)^{n-1}\} |_{SM}$ be a $(2n-1)$-dimensional volume form $\Omega_g$ on $SM$. 
In fact, $\Omega_g = \mu\, \rfloor \, (\om_g)^n$. 

Next, we introduce a $(2n-2)$-form $\Theta_g$ on $SM$ by 
\begin{eqnarray}\label{Theta_gg}
\Theta_g =_{\mathsf{def}}\; v^g\, \rfloor \, \Omega_g = (v^g \wedge \mu)\, \rfloor \, \om_g^n. 
\end{eqnarray}

Both forms, $\Omega_g$ and $\Theta_g$, are $v^g$-invariant.\smallskip

Recall the construction of the {\sf Sasaki Riemannian metric} $gg$ on the manifold $TM$, induced by a given metric $g$ on $M$ (\cite{Sa}). For any pair of tangent vectors, $v, u \in T_mM$, we consider the germ $\g_u$ of the geodesic curve through $m$ in the direction of $u$. Using the $g$-induced symmetric connection $\nabla^g$ 
on $M$, we consider the Jacobi vector field $\tilde v$ on $\g_u$, produced by the parallel transport of $v$ along $\g_u$. Using the natural parameter $s$ along $\g_u$, we get a germ at $(m, v)$ of a curve $\delta_{\g_u, v} =_{\mathsf{def}}\; \{s \to \tilde v(s)\}$ in $TM$, which projects on $\g_u$ under the map $\pi: TM \to M$. We denote by $W(v, u) = \frac{d}{ds}\delta_{\g_u, v}(0)$ the velocity vector of $\delta_{\g_u, v}$ at the point $(m, v)$. Thus, $W(v, u) \in T_{(m, v)}(TM)$. 

The correspondence  $H_{(m, v)}: u \to W(v, u)$ produces a linear injection $H_{(m, v)}: T_m(M) \to  T_{(m, v)}(TM)$. We denote by $H_{(m, v)}(TM)$ its image. Consider the {\sf the horizontal subbundle} $H(TM) \hookrightarrow TM$ of the bundle $T(TM) \to TM$ that this construction delivers. The subbundle, formed by vectors that are tangent to the fibers of $\pi: TM \to M$, is called {\sf the vertical subbundle} of $T(TM)$ and is denoted by $V(TM)$.
With its help, one splits the tangent bundle $T(TM) \to TM$ into a direct sum $H(TM) \oplus V(TM)$. 


Then, by the definition of $gg$, this bundle isomorphism $H(TM) \oplus V(TM) \approx TM \oplus TM$ is a an {\sf isometry} with respect to the metric $gg$ in the source and the metric $g \oplus g$ in the target. 

Since, under the parallel transport, the norm of the vectors $\tilde v$ is preserved, the curves $\{\delta_{\g_u, v}\}_u$ reside in $SM$, provided $(m, v) \in SM$.
Therefore, we get a bundle decomposition $T(SM) \approx H^S(TM) \oplus V(SM)$, where the $n$-bundle $H^S(TM) \to SM$ is a restriction of $H(TM) \to TM$ to $SM \subset TM$, and $V(SM) \to SM$ is the $(n-1)$-bundle, tangent to the fibers of $SM \to M$.
\smallskip

These constructions lead to the following, likely well-known, key lemma.

\begin{lemma}\label{natural_Theta} Let $(M, g)$ be a compact connected $n$-manifold with boundary.
The $(2n-2)$-form $\Theta_g := v^g \rfloor \, \Omega_g$ is equal to the form $\pm (\om_g)^{n-1}|_{SM}$ and has the properties, with respect to the involution $\tau_g: \d(SM) \to \d(SM)$, listed in (\ref{properties_of_Theta}). 
\end{lemma}

\begin{proof} 
We start the validation of the properties of $\Theta_g$ in the order they are listed in (\ref{properties_of_Theta}). 
First, we notice that $v^g \in \ker((\om_g)^{n-1}|_{SM})$, 
since $d H(w) = \om_g(v^g, w)$ and $H \equiv 1$ on $SM$. Therefore $v^g \, \rfloor \, (\beta_g \wedge (\om_g)^{n-1}) = (v^g \, \rfloor \, \b_g) \wedge (\om_g)^{n-1} = 1 \cdot (\om_g)^{n-1}$ on $SM$. 

Since $d \om_g =0$, we get that $(\om_g)^{n-1}$ is a closed form.

Using that $\Omega_g = \pm \b_g \wedge (\om_g)^{n-1}$ is a volume form on $SM$, we conclude that the form $(\om_g)^{n-1}$ does not vanish on $SM$. So we get $\dim(\ker((\om_g)^{n-1}|_{SM})) = 1$.

 In the vicinity of each point $x \in \d M$, we may pick a local coordinate system $(q^1, \dots, q^n)$ on $M$ so that $\d M$ is given by the equation $\{q^1 = 0\}$,  the gradient field $\d_{q^1}$ is $g$-orthogonal to $\d M$ and has length $1$. In other words, we may choose $q^1$ to be the $g$-induced distance function from $\d M$. The ``vertical" coordinates $(p^1, \dots, p^n) = (\d_{q^1}, \dots, \d_{q^n})$ are correlated in the standard way with $(q^1, \dots, q^n)$. Then the reflection involution is given by 
 $$\tau_g((0, q^2 \dots, q^n,\, p^1, p^2, \dots, p^n)) = (0, q^2 \dots, q^n,\, -p^1, p^2, \dots, p^n).$$

 Applying $\tau_g$ to formulas (6.3) 
 and (\ref{eq.omega}) and letting $q^1 = 0$, implies that $\tau_g^\ast(\b_g |_{\d(SM)}) = \pm \b_g|_{\d(SM)}$ and $\tau_g^\ast(\om_g |_{\d(SM)}) = \om_g|_{\d(SM)}$. 
 \smallskip
 
Let $\nu$ be the unitary vector field, normal in $gg|_{SM}$ to $\d(SM)$ in $SM$, and let $n$ be the unitary vector field, $g$-normal  to $\d M$ in $M$.   

We introduce the $(2n-2)$-volume form $\Omega^\d_g =_{\mathsf{def}} \nu \, \rfloor \,\Omega_g = \big(\nu \,\rfloor (\mu \,\rfloor \, \om_g^n)\big)$ on $\d(SM)$. Because the involution $\tau_g$ is an orientation reversing isometry on each spherical fiber of $SM |_{\d M} \to \d M$ and is an identity on the base $\d M$, we conclude that  $\tau^\ast_g(\Omega^\d_g) = -\Omega^\d_g$.

We notice that the normal vector $\nu = \nu(q, p)$ must be orthogonal in $gg$ to any vector $\theta$, tangent to the fiber of $SM \to M$ over a point $q \in \d M$. In fact, $\nu$ is tangent to the $v^g$-trajectory through the point $(q, n)$, where $n \in T_qM$ is a $g$-normal unit vector to $\d M$ in $M$ at $q \in \d M$. 
Therefore, representing $v^g$ in the local coordinates $(q, p)$ on $TM$ as $(\dot q, \dot p)$, we get $\langle \nu, v^g \rangle_{gg} =  \langle n, \dot q \rangle_{g}$. As a result, $v^g = (\dot q, \dot p)$ points inside of $SM$ if and only if $\dot q$ points inside of $M$.  Thus, the restriction of the $(2n-2)$-form $\Theta_g:=v^g \rfloor \Omega_g$ to $\d_1^\pm(SM)$ is nonnegative/nonpositive. Indeed, $\Omega^\d_g$ is the volume form on $\d(SM)$, and $\Theta_g/\Omega^\d_g = \cos(\angle_{gg}(v^g, \nu)) = \cos(\angle_{g}(\dot q, n))$ on $\d(SM)$.  
At the same time, $\tau_g^\ast(\Theta|_{\d(SM)}) =  \Theta|_{\d(SM)}$ since $$\Theta|_{\d(SM)}\, /\Omega^\d = \cos(\angle_{gg}(v^g, \nu)) = \cos(\angle_{g}(\dot q, n)),\; \text{and}$$ 
$$\tau_g^\ast(\Theta_g\,|_{\d(SM)})/\tau_g^\ast(\Omega^\d_g) =  \cos(\angle_{gg}((\tau_g)_\ast(v^g), (\tau_g)_\ast(\nu))) =\cos(\angle_{g}(\tau_g(\dot q), -n)) = \cos(\angle_{g}(\dot q, n)).$$ 
\end{proof}
%
The next lemma is an abstract of Lemma \ref{natural_Theta}; unlike the latter one, its validation is on the level of definitions.

\begin{lemma}\label{Theta_general} Let $(M, g)$ be a $n$-dimensional compact Riemannian manifold such that a volume form $\Omega$ on $SM$ is $v^g$-invariant. Then the  $(2n-2)$-form $\Theta =_{\mathsf{def}} v^g \rfloor \Omega$, is integrally dual (see Definition \ref{def.dual}) 
to the geodesic vector field $v^g$ on $SM$. So, with the help of $\Omega$, the vector field $v^g$ is nil-divergent. 

Let $N$ be a regular neighborhood of $\d(SM)$ in $SM$  and $\tilde\tau: N \to N$ a smooth involution with the properties: 
\begin{itemize}
\item $\tilde\tau(\d(SM)) = \d(SM)$,
\item $\d(SM)^{\tilde\tau} = S(\d M)$, where $S(\d M)$ denotes the spherical tangent bundle of $\d M$, 
\item $(\tilde\tau)_\ast(v_g) = \pm v_g$ along $\d(SM)$, and
$\tau^\ast(\Omega) = \pm\, \Omega$ along $\d(SM)$.
\end{itemize}
Then $\tilde{\tau}^\ast(\Theta |_{\d_1(SM)}) = \Theta|_{\d_1(SM)}$, and $\pm\Theta |_{\d_1^\pm(SM)} \geq 0$. \hfill $\diamondsuit$
\end{lemma}

The combination of Lemma \ref{Theta_general} with Lemma \ref{lem5.2} leads instantly to the following result. 

\begin{theorem}\label{main_theorem} Let $g$ be a smooth non-trapping Riemannian metric on a compact connected manifold $M$ with boundary. Let $N$ be a regular neighborhood of $\d(SM)$ in $SM$. 
Choose a $v^g$-invariant volume form $\Omega$ on $SM$ and put $\Theta = v^g \rfloor \Omega$.  Assume that a smooth involution $\tilde\tau: N \to N$ satisfies the three bullets in Lemma \ref{Theta_general} with respect to the form $\Theta$.
\smallskip

Then the form $\Theta$ is integrally dual to $v^g$. The billiard map $B_{v^g}$, the composition of the  the scattering map $C_{v^g}: \d_1^+(SM) \to \d_1^-(SM)$ with the diffeomorphism $\tau =_{\mathsf{def}} \tilde\tau|: \d_1^-(SM) \to \d_1^+(SM)$, preserves the $\Theta$-induced measure $\mu_\Theta$ on $ \d_1^+(SM)$. \smallskip

There exists a $v^g$-harmonizing metric $g^\bullet$ on $SM$ so that $\Theta = v^g \rfloor \Omega$ coincides with $\ast_{g^\bullet}(dF)$. 
\hfill $\diamondsuit$
\end{theorem}

\begin{corollary}\label{infinite_return_for_billiard}
Under the assumptions of Theorem \ref{main_theorem}, the Billiard Map $B_{v^g,\, \tau} = \tau \circ C_{v^g}: \d_1^+(SM) \to \d_1^+(SM)$ has the infinite return property (as in Theorem \ref{proto-billiard}). \hfill $\diamondsuit$
\end{corollary}
\smallskip

Viewing the fundamental $1$-form $\b_g$ as a smooth section of the cotangent bundle $T^\ast(SM) \to SM$, we denote by $\b_g^{\d^+}$ the restriction of $\b_g$ to $\d_1^+(SM) \subset SM$.\smallskip

The next theorem claims that, for non-trapping metrics, the scattering map $C_{v^g}$ allows to reconstruct the canonical {\sf contact structure} on $SM$, provided we know it along the boundary $\d(SM)$.

\begin{theorem}\label{eq.beta} Let $g$ be a boundary generic and non-trapping metric on a compact connected manifold $M$ with boundary. Assume that any geodesic curve in $M$ is either transversal to $\d M$ at \emph{some} point or is simply tangent to $\d M$ at \emph{some} point \footnote{This assumption is valid if each connected component of $\d M$ is strictly concave or convex in $g$.}.

We choose a well-balanced Lyapunov function $F: SM \to \R$.  
\smallskip

Then the scattering map $C_{v^g}: \d_1^+(SM) \to \d_1^-(SM)$, together with the function $F^\d: \d(SM) \to \R$ and the $1$-form $\b_g^{\d^+}$, allow for a reconstruction of $SM$, the geodesic field $v^g$, and the contact $1$-form $\b_g$, up to a diffeomorphism $\Phi: SM \to SM$ that is the identity on $\d(SM)$ and preserves the $1$-form  $\b_g^{\d^+}$.
\end{theorem}

\begin{proof} The form $\b_g$ on $SM$ is $v^g$-invariant. Therefore, $\b_g^{\d^+}$ spreads uniquely by the geodesic flow $\{x \to \phi^\theta(x)\}_{\theta \in \R}$ along each $v^g$-trajectory $\tilde\g_x$ that passes through a point $x \in \d_1^+(SM)$. Let $U(\tilde\g_x)$ be a flow-invariant tube around $\tilde\g_x$ and $a, b \in \tilde\g_x$. Because $F$ is well-balanced, 
the $v^g$-flow $\phi$ is delivered by the family of local diffeomorphisms $\phi_{a, b}: U(\tilde\g_x) \to U(\tilde\g_x)$ 
 that map the constant  level set $\hat F^{-1}(F(a))$ to the constant  level set $\hat F^{-1}(F(b))$. Therefore, the grid in $SM$, formed by the pair of transversal foliations $\mathcal F(v^g) = \{\tilde\g_x\}_{x \in \d_1^+(SM)}$ and $\mathcal G(F) = \{F^{-1}(c)\}_{c \in F(\d(SM))}$, allow not only to reconstruct the pair $(SM, \mathcal F(v^g))$, but also the vector field $v^g$ and the 
geodesic flow. Moreover, by the argument above, $\b_g$ may also be recovered from these two foliations and the section $\b_g^{\d^+}$. In turn, the two foliations are determined by the pairs $\big(x \in \d_1^+(SM),\; c \in F(\d(SM))\big)$.

Thus, for a well-balanced Lyapunov function $F: SM \to \R$, the boundary confined data of  the theorem make it possible to reconstruct, up to a diffeomorphism of $SM$, which is constant on its boundary, the standard contact structure $\ker(\b_g)$ on $SM$ and, since $\om_g|_{SM} = d\b_g$, to reconstruct the restriction $\om_g|_{SM}$ of the symplectic $2$-form $\om_g$. 
\end{proof}

\section{On the averages of ergodic billiards for non-trapping metrics}

The chaotic dynamics of billiard maps iterations has been a subject of a well-established and flourishing research industry (for example, see \cite{Si}, \cite{Si1}, \cite{BFK}, and \cite{ChM}). Since the Euclidean billiards with {\sf strictly concave boundary} are {\sf dispersing}, the reader may replace in what follows any assumption about the ergodicity of a billiard $(M, g)$ by the assumption that billiard is Euclidean and its boundary is strictly concave \cite{ChM}, \cite{BFK}.

Let us recall few standard definitions, related to the dynamics of measure-preserving maps, as they apply to the billiard maps.

\begin{definition}\label{def9.7}
 The billiard map $B_{v^g}: \d_1^+(SM)(v^g) \to \d_1^+(SM)(v^g)$ is said to be \textsf{ergodic} with respect to a given $n$-form $\Theta$ on $SM$ as in Theorem  
\ref{main_theorem}, if the invariance of a Lebesgue-measurable set $K \subset \d_1^+(SM)$ under  the billiard map $B_{v^g}$ implies that ether the measure $\mu_\Theta(K)$ or the measure $\mu_\Theta(\d_1^+(SM) \setminus K)$ is zero.  \hfill $\diamondsuit$
\end{definition}

A natural inclination is to anticipate that, for a ``random" pair $(M, g)$, $g$ being non-trapping,  boundary generic, and ``bumpy", the billiard map $B_{v^g}$ is ergodic. This thinking is supported by the theory of {\sf Sinai billiards}. For them, the smooth boundary $\d M$ is strictly concave, so that the billiard is dispersing, which implies not only ergodicity of the billiard map, but also its mixing property and positive entropy \cite{Si}, \cite{Si1}.
\smallskip

Let us recall now the content of the famous Birkhoff Theorem \cite{Bi}, as it applies to the environment of the billiard maps, with the measure $\mu_\Theta$ on $\d_1^+(SM)$ being induced by the appropriate   $(2n-2)$-form $\Theta$ on $SM$. We assume that $\Theta |_{\d(SM)}$ is invariant under an involution $\tau: \d(SM) \to \d(SM)$, a generalized billiard reflection; so the billiard map $B_{v^g}: \d_1^+(SM) \to \d_1^+(SM)$ preserves the measure $\mu_\Theta$.   

Let $f \in L^1(\d_1^+(SM), \mu_\Theta)$ be an integrable (with respect to the measure $\mu_\Theta$) function $f: \d_1^+(SM) \to \R$. 
Its {\sf time average} is defined by the formula
$$\hat f(z) =_{\mathsf{def}} \lim_{m \to \infty}\; \frac{1}{m}\, \sum_{k=0}^{m-1} f\big((B_{v^g})^{\circ k}(z)\big).$$
In fact, the limit $\hat f(z)$ exists almost\footnote{the exceptional $z$'s form a set of $\mu_\Theta$-measure $0$.} for all $z \in \d_1^+(SM)$, and $\hat f \in L^1(\d_1^+(SM),\mu_\Theta)$.
Moreover, $\hat f$ is an invariant function (that is, if $B_{v^g}\circ \hat f = \hat f$ almost everywhere) and  
$$\int_{\d_1^+(SM)} \hat f \, d\mu_\Theta = \int_{\d_1^+(SM)} f \, d\mu_\Theta.$$
\smallskip

The {\sf space average} of $f$ is defined as 
$$\bar f =_{\mathsf{def}} \frac{1}{\mu_\Theta(\d_1^+(SM))} \int_{\d_1^+(SM)} f \, d\mu_\Theta.$$

In particular, if the billiard map $B_{v^g}$ is {\sf ergodic}, then $\hat f$ must be  \emph{constant} almost everywhere: indeed, any level set $\hat f^{-1}((-\infty, c))$ is $B_{v^g}$-invariant for any $c \in \R$. As a result, $\hat f = \bar f$ almost everywhere (for example, see \cite{W}). 

Therefore, using that $\mu_\Theta(\mathcal T(v^g)) =_{\mathsf{def}} \mu_\Theta(\d_1^+(SM))$ and that, by Theorem \ref{main_theorem}, $B_{v^g}$ is a $\mu_\Theta$-measure-preserving transformation, we get the following version of the Birkhoff Theorem:

\begin{theorem}\label{ergodic_billiard} Let $(M, g)$ be a compact connected smooth Riemannian $n$-manifold with boundary, the metric $g$ being non-trapping and boundary generic. 
Let a $(2n-2)$-form $\Theta$ on $SM$ be as in Theorem \ref{main_theorem} or in the list (\ref{properties_of_Theta}). 
 \smallskip

If the billiard map $B_{v^g}$ is ergodic in the measure $\mu_\Theta$, then 
\begin{eqnarray}\label{eqBIRKHOFF}
\lim_{m \to \infty}\; \frac{1}{m}\, \sum_{k=0}^{m-1} f\big((B_{v^g})^{\circ k}(z)\big) =  \frac{1}{\mu_\Theta(\mathcal T(v^g))} \int_{\d_1^+(SM)} f \, d\mu_\Theta
\end{eqnarray}
for any given function $f \in L^1(\d_1^+(SM), \mu_\Theta)$ and almost all points $z \in \d_1^+(SM)$.
\hfill $\diamondsuit$
\end{theorem}



For ergodic billiard maps $B_{v^g}$, some ``\emph{metric-flavored}" holographic properties hold: as Theorem \ref{balanced F-variation} and Corollary \ref{cor.FXT} below claim, for boundary generic non-trapping metrics $g$ on $M$, the $g$-induced volume of the space $M$ can be recovered from the volume of the space of geodesics and the {\sf average length} $\mathsf{av}(\ell)$ of a {\sf free geodesic}. A free geodesic in $M$ is a segment $[\g]$ of a geodesic curve $\g$ in $M$ such that the boundary $\d [\g] \subset \d M$ and $int([\g]) \cap \d M = \emptyset$.\smallskip

In our treatment, we are guided by the Santalo formula \cite{S}, and 
\cite{Ch}, 
 \cite{ChM},
\cite{Tab}. \smallskip

For flat compact billiards $M$ in the Euclidean space $\mathsf E^n$ or in the flat torus $\mathsf T^n$, the average length $\mathsf{av}(\ell)$ of a free trajectory is given  by a beautiful formula (see \cite{S}, \cite{Ch}):
\begin{eqnarray}\label{eq9.8}
\mathsf{av}(\ell) = \frac{vol_\mathsf E(S^{n-1})}{vol_\mathsf E(B^{n-1})} \cdot \frac{vol_\mathsf E(M)}{vol_\mathsf E(\d M)} = 2\sqrt{\pi} \cdot \frac{\Gamma(\frac{n+1}{2})}{\Gamma(\frac{n}{2})} \cdot \frac{vol_\mathsf E(M)}{vol_\mathsf E(\d M)},
\end{eqnarray}
in terms of the Euclidean volume $vol_\mathsf E(B^{n-1})$ of the unit ball $B^{n-1}$ and of the  volume $vol_\mathsf E(S^{n-1})$ of the unit sphere $S^{n-1} \subset \mathsf E^n$.\smallskip

\smallskip

These results admit the following generalizations.

\begin{theorem}\label{F-variation}  Let $(M, g)$ be a compact connected Riemannian $n$-manifold with boundary, such that the metric $g$ is boundary generic and non-trapping. 
Let $F: SM \to \R$ be a Lyapunov function for the geodesic field $v^g$. 
 For each $z \in  \d_1^+(SM)$, consider the variation $$\D_F(z) =_{\mathsf{def}} F(C_{v^g}(z)) - F(z)$$ of $F$ along the segment $[z, C_{v^g}(z)]$ of the $v^g$-trajectory $\tilde\g_z$ through $z$. 
 This construction gives rise to a well-defined measurable function $\D_F: \d_1^+(SM) \to \R.$

Consider a differential $(2n-2)$-form $\Theta$ with the properties as in (\ref{properties_of_Theta}).  \smallskip

With these ingredients in place, the following statements hold:
\begin{itemize}
\item The average of the variation of $F$ along the $v^g$-trajectories can be calculated via the Santalo-type formula
\begin{eqnarray}\label{eq9.9}
\mathsf{av}(\D_F) =_{\sf{def}}  \; \frac{\int_{\d_1^+(SM)} \D_F \cdot \Theta}{\int_{\d_1^+(SM)} \Theta}\; = \, \frac{\int_{SM} dF \wedge \Theta}{\int_{\d_1^+(SM)} \Theta} \; =_{\sf{def}}\; \frac{vol_{dF \wedge \Theta}(SM)}{vol_{\Theta}\big(\mathcal T(v^g)\big)}.
\end{eqnarray}
\item If the billiard map $B_{v^g}$ is ergodic with respect to the measure $\mu_{\Theta}$ (for example, if $g$ is Euclidean and $\d M$ is strictly concave in $g$), then the average of the variation of $F$ can be also calculated via the formula   
\begin{eqnarray}\label{average_under_iterrations}
\mathsf{av}(\D_F) \; = \; \lim_{m \to \infty}\; \frac{1}{m}\, \sum_{k=0}^{m-1} \D_F\big((B_{v^g})^{\circ k}(z)\big)  
\end{eqnarray}
for almost all $z \in \d_1^+(SM)$. 
\end{itemize}
\end{theorem}

\begin{proof} If the metric $g$ is non-trapping, the vector field $v^g$ is traversing. So there exists a Lyapunov function $F: SM \to \R$ for the geodesic field $v^g$. 

By (\ref{properties_of_Theta}), the billiard map $B_{v^g}$ preserves the measure $\mu_\Theta$ on $\d_1^+(SM)$, and $dF \wedge \Theta$ may serve as the volume form on $SM$ (not to be confused with the $v^g$-invariant volume form $\Omega$ that gave rise to $\Theta$ via the formula $\Theta =_{\mathsf{def}} v^g \rfloor \Omega$ !). 


Since the scattering map $C_{v^g}$ is continuous and smooth away from a set of vanishing Lebesgue measure, the function $\D_F$ is a Lebesgue-measurable, and thus $\mu_{\Theta}$-measurable. 

Thus, using that $d\Theta =0$ and by the Stokes' formula, we get 
$$\int_{\d_1^+(SM)}\; \D_F \cdot \Theta = \int_{ \d_1^+(SM)}^{\d_1^-(SM)}\; F \cdot \Theta \; \stackrel{Stokes}{=}\; \int_{SM} dF \wedge \Theta =_{\sf{def}} vol_{dF \wedge \Theta}(SM)$$

By definition, $vol_{\Theta}\big(\mathcal T(v^g)\big)$---the volume of the space of geodesics---is the integral $\int_{\d_1^+(SM)} \Theta$. Therefore, 
\begin{eqnarray}\label{eq9.11}
\mathsf{av}(\D_F)  \; = \;\frac{\int_{\d_1^+(SM)} \D_F \cdot \Theta}{\int_{\d_1^+(SM)} \Theta} \; = \frac{\int_{SM} dF \wedge \Theta}{\int_{\d_1^+(SM)} \Theta} \; =_{\sf{def}} \; \frac{vol_{dF \wedge \Theta}(SM)}{vol_{\Theta}\big(\mathcal T(v^g)\big)}, \nonumber
\end{eqnarray}
which validates formula (\ref{eq9.9}).
\smallskip

When $B_{v^g}$ is ergodic in $\mu_{\Theta}$, by  applying Theorem \ref{ergodic_billiard}  to the function $f =_{\mathsf{def}} \D_F$, we prove formula (\ref{average_under_iterrations}). \hfill
\end{proof}

For traversing flows $v^g$, their Lyapunov functions $F$ attend extrema on the boundary $\d(SM)$. Hence, the global variation $\mathsf{var}(F)$ of $F$ on $SM$ is equal to its variation $\mathsf{var}(F^\d)$ on $\d(SM)$. On the other hand, evidently, $\D_F(z) \leq \mathsf{var}(F)$ for all $z$. This leads to the following direct corollary of Theorem \ref{F-variation}.

\begin{corollary} Let $(M, g)$ be a compact connected Riemannian $n$-manifold with boundary, such that the metric $g$ is boundary generic and non-trapping. Consider the variation $\mathsf{var}(F^\d) =_{\mathsf{def}} \mathsf{var}(F|_{\d(SM)})$ of the Lyapunov function $F: SM \to \R$ on the boundary $\d(SM)$.\smallskip

Then, for any $\Theta$ as in Theorem \ref{F-variation},
$$\mathsf{var}(F^\d) \cdot vol_\Theta(\mathcal T(v^g))\, \geq \, vol_{dF \wedge \Theta}(SM).$$
In particular, if $\mathsf{var}(F^\d)= 1$, then $vol_\Theta(\mathcal T(v^g))\, \geq \, vol_{dF \wedge \Theta}(SM)$.
\hfill $\diamondsuit$
\end{corollary}

\begin{corollary} Under the ergodicity hypotheses from the second bullet of Theorem \ref{F-variation}, knowing the $\Theta$-induced volume of the space of geodesics $\mathcal T(v^g)$, together with the limit $$\lim_{m \to \infty}\; \frac{1}{m}\, \sum_{k=0}^{m-1} \D_F\big((B_{v^g})^{\circ k}(z)\big)$$ for almost any  $z \in \d_1^+(SM)$, allows to determine the $(dF \wedge \Theta)$-induced volume of the space $SM$. 

Both $vol_\Theta(\mathcal T(v^g))$ and the limit above may be recovered in terms of the data, confined to the boundary $\d(SM)$.
\hfill $\diamondsuit$
\end{corollary}

Now we will derive two theorems, Theorem \ref{balanced F-variation} and Theorem \ref{ell(gg)}, that are in the spirit of Theorem \ref{F-variation}, but require some additional analysis. 
\smallskip

We denote by $gg$ the restriction of the Sasaki metric \cite{Sa} on $TM$ to $SM$, and by $\ell_{gg}(z)$ the length (in the metric $gg$) of the segment $[z, C_{v^g}(z)]$ of the $v^g$-trajectory $\tilde\g_z$ through $z \in  \d_1^+(SM)$. In fact, $\ell_{gg}(z)$ is the length of the free geodesic segment $\pi([z, C_{v^g}(z)])$ in the metric $g$ on $M$, where $\pi: SM \to M$ is the obvious projection. 

In general, for any metric $g^\bullet$ on $SM$, we denote by $\ell_{g^\bullet}(z)$ the length, in the metric $g^\bullet$, of the segment $[z, C_{v^g}(z)]$ of the $v^g$-trajectory $\tilde\g_z$ through $z \in  \d_1^+(SM)$. We denote by $\ell_{g^\bullet}: \d_1^+(SM) \to \R_+$ the measurable function that this construction generates.
\smallskip

With these notations fixed, for the classical billiard maps based on the elastic reflections, we get a stronger than Theorem \ref{F-variation} claim. 

\begin{theorem}\label{balanced F-variation}  Let $(M, g)$ be a compact connected Riemannian $n$-manifold with boundary, such that the metric $g$ is boundary generic and non-trapping. 

Let $F: SM \to \R$ be a Lyapunov function for the geodesic vector field $v^g$. 
Consider the differential $(2n-2)$-form  $\Theta_g =_{\mathsf{def}}\, \om_g^{n-1}$, where $\om_g$ is the restriction of the symplectic form on $TM$ to $SM$.\smallskip

With these ingredients in place, the following statements hold: 
\begin{itemize}
\item The measure $\mu_{\Theta_g}$ on $\d_1^+(SM)$,  defined by the closed $(2n-2)$-form $\Theta_g$ with the help of formula (\ref{eq2.2}), is invariant under the billiard map $B_{v^g}$.\smallskip

\item There exists a $v^g$-harmonizing metric $g^\bullet$ on $SM$ such that $\ast_{g^\bullet}(dF) = \Theta_g$. Moreover, in $g^\bullet$, the $v^g$-trajectories are still geodesic curves, and the leaves of the foliation $\mathcal G = \{F^{-1}(c)\}_{c \in \R}$ are volume minimizing (in their relative to $\d(SM) \cap F^{-1}(c)$ class) hypersurfaces. \smallskip

\item 
The average value of the length function $\ell_{g^\bullet}$ on the space of $v^g$-trajectories  can be calculated via the formula
\begin{eqnarray}\label{AA}
\mathsf{av}(\ell_{g^\bullet}) \; =_{\mathsf{def}} \frac{\int_{\d_1^+(SM)} \ell_{g^\bullet} \cdot \om_g^{n-1}}{\int_{\d_1^+(SM)} \om_g^{n-1}} \; \nonumber \\ = \,  \frac{\int_{SM} dF \wedge \om_g^{n-1}}{\int_{\d_1^+(SM)} \om_g^{n-1}}  = \frac{vol_{g^\bullet}(SM)}{vol_{\om_g^{n-1}}(\mathcal T(v^g))}.
\end{eqnarray}

\item If the billiard map $B_{v^g}$ is ergodic with respect to the measure $\mu_{\om^{n-1}}$ (say, if $\d M$ is strictly concave in $g$), 
then the average value of the function $\ell_{g^\bullet}$ can be also calculated via the formula   
\begin{eqnarray}\label{B}
\mathsf{av}(\ell_{g^\bullet}) \; = \; \lim_{m \to \infty}\; \frac{1}{m}\, \sum_{k=0}^{m-1} \ell_{g^\bullet}\big((B_{v^g})^{\circ k}(z)\big)  
\end{eqnarray}
for almost all $z \in \d_1^+(SM)$. 

\end{itemize}
\end{theorem}

\begin{proof} Our argument is based essentially on Lemma \ref{natural_Theta}. Thanks to that lemma, the form $\pm\b_g \wedge \om_g^{n-1}|_{SM}$ is the volume form $\Omega_{gg}$ on $SM$, and the form $\Theta_g = v^g \rfloor \Omega_{gg} = \om_g^{n-1}|_{SM}$ on $SM$ has all the desired properties from the list (\ref{properties_of_Theta}). In particular, the form $\Theta_g  = \om_g^{n-1}$ produces a measure $\mu_{\om_g^{n-1}}$ on $\d_1^+(SM)$, and by Theorem \ref{main_theorem}, the billiard map $B_{v^g}: \d_1^+(SM) \to \d_1^+(SM)$ preserves this measure. \smallskip

The second bullet is validated by Theorem \ref{foliations on SM}, which, in particular,  claims that there exists a $v^g$-harmonizing metric $g^\bullet$ on $SM$ such that $\ast_{g^\bullet}(dF) = \om_g^{n-1}$ and $\ker(dF) \perp_{g^\bullet} \ker(\om_g^{n-1})$. Thus $dF \wedge \om_g^{n-1}$ is the volume form $dg^\bullet$, produced by $g^\bullet$. 
In the metric $g^\bullet$, by its construction, we have $\|dF\|^\ast_{g^\bullet} = 1$; hence $dF(v^g) = \| v^g\|_{g^\bullet}$. 
\smallskip




Now we will validate formulas (\ref{AA}) in the third bullet. Since $\|dF\|^\ast_{g^\bullet} = 1$, the variation $\D_F(z) := F(C_{v^g}(z)) - F(z) = \ell_{g^\bullet}(z)$. 
So we get 
$$\int_{\d_1^+(SM)}\; \ell_{g^\bullet} \cdot \om_g^{n-1} = \int_{\d_1^+(SM)}\; \D_F \cdot \om_g^{n-1} = $$ 
$$ = \int_{ \d_1^+(SM)}^{\d_1^-(SM)}\; F \cdot \om_g^{n-1}\; \stackrel{Stokes}{=}\; \int_{SM} dF \wedge \om_g^{n-1}  = vol_{g^\bullet}(SM).$$

Finally, formula (\ref{B}) in the last bullet follows by applying (\ref{eqBIRKHOFF}) to $f=\ell_{g^\bullet}$ and the billiard map.
\end{proof}



\begin{remark}\emph{There is a tension between Theorem \ref{balanced F-variation}, which deals with {\sf taught foliations} $\mathcal G$ on $SM$, and Theorem \ref{F-variation}, which deals with {\sf fillable contact structures} induced by $\b_g$ on $SM$ (see \cite{ET} for the relevant definitions).} \hfill $\diamondsuit$
\end{remark}
\smallskip

Let us revisit the classical case of the Sasaki metric $gg$ on $SM$  (Cf. \cite{S}, \cite{Ch}).

\begin{theorem}\label{ell(gg)} Let $(M, g)$ be a compact connected Riemannian $n$-manifold with boundary, such that the metric $g$ is boundary generic and non-trapping. We denote by $gg$ the restriction of the Sasaki metric on $TM$ to $SM$.
 
Consider the measurable function $\ell_{gg}: \d_1^+(SM) \to \R_+$, defined as the length in $gg$ of the segment $[z, C_{v^g}(z)]$ of the $v^g$-trajectory $\tilde\g_z$ through $z \in \d_1^+(SM)$.\footnote{$\ell_{gg}(z)$ equals the length $\ell_g(z)$ of the free geodesic segment $\pi([z, C_{v^g}(z)])$ in the metric $g$ on $M$, where $\pi: SM \to M$ is the obvious projection.} 

Let $\Theta_g =_{\mathsf{def}}\, \om_g^{n-1}$, where $\om_g$ is the restriction of the symplectic form on $TM$ to $SM$.\smallskip

With these ingredients in place, the following statements hold:
\begin{itemize}
\item The measure $\mu_{\Theta_g}$ on $\d_1^+(SM)$, defined by the form $\Theta_g$ with the help of formula (\ref{eq2.2}), is invariant under the billiard map $B_{v^g}$.
\smallskip

\item The average value of the function $\ell_{gg}$, which is equal to the average length of a free geodesic segment in $M$, can be calculated via the formula
\begin{eqnarray}\label{ell_gg}
\mathsf{av}(\ell_{gg})  =_{\sf{def}} \frac{\int_{\d_1^+(SM)} \ell_{gg} \cdot \om_g^{n-1}}{\int_{\d_1^+(SM)}  \om_g^{n-1}} =
\frac{\int_{SM} \b_{g} \wedge \om_g^{n-1}}{\int_{\d_1^+(SM)} \om^{n-1}_g} \nonumber \\  = \frac{vol_{gg}(SM)}{vol_{\om_g^{n-1}}(\mathcal T(v^g))}  \; =  \frac{vol_{\mathsf E}(S^{n-1})}{vol_{\mathsf E}(B^{n-1})}\cdot \frac{vol_g(M)}{vol_g(\d M)}\nonumber \\  = 2\sqrt{\pi} \cdot \frac{\Gamma(\frac{n+1}{2})}{\Gamma(\frac{n}{2})} \cdot \frac{vol_g(M)}{vol_g(\d M)}.
\end{eqnarray}

\item If the billiard map $B_{v^g}$ is ergodic with respect to the measure $\mu_{\om^{n-1}}$, then $\mathsf{av}(\ell_{gg})$ can be also calculated, for almost all $z \in \d_1^+(SM)$, via the formula   
\begin{eqnarray}\label{B}
\mathsf{av}(\ell_{gg}) \; = \; \lim_{m \to \infty}\; \frac{1}{m}\, \sum_{k=0}^{m-1} \ell_{gg}\big((B_{v^g})^{\circ k}(z)\big).  
\end{eqnarray}

\end{itemize}
\end{theorem}

\begin{proof} 
As in Theorem \ref{balanced F-variation}, the argument is based  on Lemma \ref{natural_Theta}. Thanks to the lemma,  $\pm\b_g \wedge \om_g^{n-1}|_{SM}$ is the Sasaki 
volume form $\Omega_{gg}$ on $SM$, the form $\Theta_g = v^g \rfloor \Omega_{gg} = \om_g^{n-1}|_{SM}$ on $SM$ has all the desired properties from the list (\ref{properties_of_Theta}). In particular, the form $\Theta_g  = \om_g^{n-1}$ produces the measure $\mu_{\om_g^{n-1}}$ on $\d_1^+(SM)$, and, by Theorem \ref{main_theorem}, the billiard map $B_{v^g}: \d_1^+(SM) \to \d_1^+(SM)$ preserves this measure.\smallskip

Let $\nu$ be the inner unitary normal to $\d(SM)$ in $SM$ (with respect to $gg$) vector field. Let $\Omega^\d_{gg} =_{\mathsf{def}}\, \nu \, \rfloor \, \Omega_{gg}$ be the Sasaki volume form on $\d(SM)$.\smallskip

Recall that any $v^g$-trajectory $\tilde\g \subset SM$ is a geodesic curve in the metric $gg$ (\cite{Be}, Proposition 1.106). By the definition of the Sasaki metric $gg$, $\tilde\g$ is horizontal (i.e., tangent to the horizontal distribution $H(TM)$ on $TM$, defined by the $g$-induced connection) and orthogonal to each spherical fiber of the fibration $\pi: SM \to M$. Therefore,  $\pi$ projects $\tilde\g$ onto a geodesic curve $\g \subset M$ in the metric $g$ and the lengths of the two geodesics are equal: $\ell_{gg}(\tilde\g) =  \ell_g(\g)$.

Consider the discontinuous map $\Pi: SM \to \d_1^+(SM)$ that takes any point $x \in SM$ to the maximal point $\Pi(x)$ in the finite set $\tilde\g_x \cap \d_1^+(SM)$, so that $\Pi(x)$  lies below $x$ on the trajectory $\tilde\g_x$ (the order in $\tilde\g_x$ is defined by $v^g$). Away from the zero measure set $\Pi^{-1}\big(\d_2(SM)(v^g)\big)$, the map $\Pi$ is a smooth fibration with the fiber being a closed segment. Thus we may integrate the $1$-form $\b_g$ along the $\Pi$-fibers. This integration leads to the following equations: 
\begin{eqnarray}\label{volSM}
\int_{\d_1^+(SM)}\; \ell_{gg} \cdot \om_g^{n-1} = \int_{ \d_1^+(SM)}\; \Big(\int_{z}^{C_{v^g}(z)} \b_g\Big) \cdot \om_g^{n-1} \nonumber \\ 
\stackrel{Fubini}{=} \int_{SM} \b_g \wedge \om_g^{n-1} = \pm\, vol_{gg}(SM).
\end{eqnarray}

The equality marked ``Fubini" is a special case of the generalized Fubini formula, proven by Dieudonn\'{e} \cite{D}. Under the name of ``projection formula" it can be found in \cite{BT}, Proposition 6.15.\footnote{See \cite{Her} or \cite{S}, page 349, formulas (19.64a) and (19.64b), for its generalization.} 
In our context, the projection formula uses that $v^g \in \ker(\om_g^{n-1})$, $\b_g(v^g) = 1$, and $\Theta|_{\ker(\b_g)} > 0$.
\smallskip

The metric $gg$ induces the Euclidean metric $g_{\mathsf E}$ on the fibers of the fibration $\pi: TM \to M$. Let $\mathcal H_g =_{\mathsf{def}} H(SM)$ denote the horizontal $n$-distribution on $SM$. It is the intersection with $SM$ of the horizontal $n$-distribution $H(TM)$ on $TM$ that is defined by the $g$-induced connection on $M$ (as described in paragraphs that preceded Lemma \ref{natural_Theta}). Recall that $v^g \in \mathcal H_g$.

Let $\rho_{\mathsf E}(S^{n-1})$ denote the $(n-1)$-form on $SM$ that coincides with the Euclidean volume form on the $\pi$-fibers $\{S^{n-1}\}$ and vanishes on the polyvectors which contain vectors from the horizontal distribution $\mathcal H_g$. 

 Therefore, by another instance of the Fubini formula, being applied to the forms $\rho_{\mathsf E}(S^{n-1})$ and $\pi^\ast(dg)$, we get $$vol_{gg}(SM) =  vol_{\mathsf E}(S^{n-1})\cdot vol_g(M).$$ 

The manifold $\d_1^+(SM)$ fibers over the boundary $\d M$ with the fiber being a hemisphere $S^{n - 1}_+ \subset S^{n-1} \subset \mathsf E^{n}$. 


As in the proof of Lemma \ref{natural_Theta}, the restriction of the form $\Theta_g = \om_g^{n-1}$ on the boundary $\d(SM)$, is proportional to the $gg$-induced volume $(2n-2)$-form $\Omega^\d_{gg} = \pm\, \nu \, \rfloor \, (\b_g \wedge \om_g^{n-1})$ on $\d(SM)$ with the coefficient of proportionality $\cos(\angle(v^g, \nu)) = \langle v^g, \nu \rangle_{gg}$. This function $\langle v^g, \nu \rangle_{gg} \geq 0$ exactly on the locus $\d_1^+(SM)$.\smallskip

By definition, $vol_{\Theta_g}(\mathcal T(v^g)) =  \mu_{\Theta_g}(\d_1^+(SM))$. 
Therefore, the $gg$-induced volume of the space of geodesics is given by 
 $$vol_{\Theta_g}(\mathcal T(v^g)) =_{\mathsf{def}} \int_{\d_1^+(SM)} \; \Theta_g = \int_{\d_1^+(SM)} \;\langle v^g, \nu \rangle_{gg} \cdot \Omega^\d_{gg}.$$
 
 Consider the $gg$-orthogonal decomposition $T(SM) \approx H(SM) \oplus V(SM)$ of the tangent bundle into horizontal and vertical distributions on $SM$.
 Then both vectors, 
 $v^g$ and the normal to $\d(SM)$ vector $\nu \in T_{(m, v)}(SM)$ at the point $(m, v) \in \d(SM)$, are horizontal. We can interpret the vector $\nu$ as a inner normal vector to the hemisphere $S_+^{n-1} \subset B^n \subset T_{(m, v)}(SM)$, pointing towards the center of the unit ball $B^n$. Then $\langle v^g, \nu \rangle_{gg} = \langle v, \nu \rangle_g$. 

Put $dg^\d  =_{\mathsf{def}} d g|_{\d M}$. Examining the fibration $\pi: \d_1^+(SM) \to \d M$ with the fiber $S_+^{n-1}$, we  get that the volume form $\Omega^\d_{gg} = \pi^\ast(dg^\d) \wedge d_{\mathsf E}(S_+^{n-1})$. 
Note that $$\langle v, \nu \rangle_g \cdot d_{\mathsf E}(S_+^{n-1}) = d_{\mathsf E}(B^{n-1}),$$ the Euclidean volume form on the equatorial ball $B^{n-1} \subset B^n$. Therefore, 
$$\langle v^g, \nu \rangle_{gg}\cdot \Omega^\d_{gg} = \pi^\ast(dg^\d) \wedge \big(\langle v, \nu \rangle_g \cdot d_{\mathsf E}(S_+^{n-1})\big) = \pi^\ast(dg^\d) \wedge d_{\mathsf E}(B^{n-1}).$$
Let $D^\d(SM)$ denote the space of the unit tangent disk bundle over $\d(SM)$. Again, by the Fubini theorem, applied to the fibration $\pi^\d: D^\d(SM) \to \d M,$ we get
$$vol_{\Theta}(\mathcal T(v^g)) =  \int_{\d M}\Big(\int_{{(\pi^\d)}^{-1}(m)} d_{\mathsf E} \Big)\; d g^\d = vol_{\mathsf E}(B^{n-1})\cdot vol_g(\d M).$$
This proves formula (\ref{eq9.11}).

Thus formula (\ref{ell_gg}) follows from formula (\ref{volSM}): 
\begin{eqnarray}
\mathsf{av}(\ell_{gg}) = \frac{\int_{SM} \b_g \wedge \om_g^{n-1}}{\int_{\d_1^+(SM)} \om^{n-1}_g} = \frac{vol_{gg}(SM)}{\int_{\d_1^+(SM)} \om_g^{n-1}}  \; =  \frac{vol_{\mathsf E}(S^{n-1})}{vol_{\mathsf E}(B^{n-1})}\cdot \frac{vol_g(M)}{vol_g(\d M)}. \nonumber
\end{eqnarray} 

Finally, (\ref{B}) follows from Theorem \ref{ergodic_billiard}.
\hfill 
\end{proof}
\smallskip


\begin{definition} Let  $(M, g)$ be a compact Riemannian manifold with boundary, where $g$ is a non-trapping metric.  We denote by $\g_{x, y}$ a free geodesic arc that connects a pair of points $x, y \in M$ and whose interior belongs to the interior of $M$,\footnote{Due to possible concavity of $\d M$, not any pair $x, y \in M$ can be connected by such a geodesic.} and by $\ell_g(\g_{x, y})$  its length.
Consider the length of the longest free geodesic segment in $M$:
$$\mathsf{gd}(M, g) =_{\sf{def}}\;  \sup_{\{x, y \in M\}} \{\ell_g(\g_{x, y})\} = \sup_{\{x, y \in \d M\}} \{\ell_g(\g_{x, y})\},$$
 and call it the {\sf geodesic diameter} of $M$.
\hfill $\diamondsuit$
\end{definition}

\begin{figure}[ht]\label{figBB}
\centerline{\includegraphics[height=1.2in,width=3in]{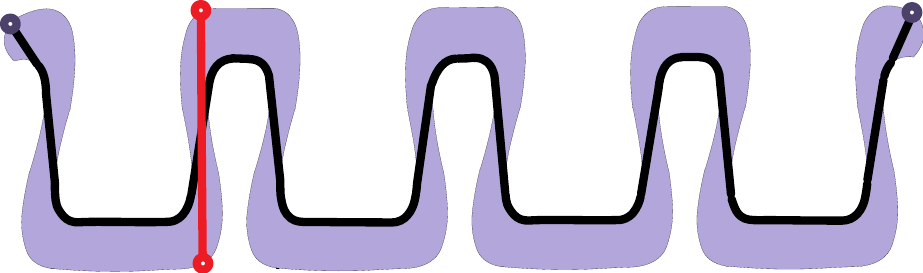}}
\bigskip
\caption{\small{The intrinsic diameter (the length of the wavy black curve) and the geodesic diameter $\mathsf{gd}(M, g)$ (the length of the red segment) of a Riemannian manifold $M \subset \mathbb E^2$ with boundary. The Riemannian metric $g$ on $M$ is Euclidean.}}  
\end{figure}

The relation between the geodesic diameter $\mathsf{gd}(M, g)$ and the ordinary diameter $\mathsf{d}(M, g)$ is subtle due to the boundary effects, as shown in Fig.3. 
\smallskip

If $g$ is trapping, then $\mathsf{gd}(M, g) = +\infty$, while $\mathsf{d}(M, g) < \infty$. At the same time, if $(M, g)$ is geodesically convex and a non-trapping, then $\mathsf{gd}(M, g)= \mathsf{d}(M, g)$.
\smallskip

Using (\ref{ell_gg}), we get the following inequality: 

\begin{corollary}\label{diameter} For any non-trapping metric $g$ on a compact manifold $M$ with boundary, 
\begin{eqnarray}\label{isoM}
vol_g(M)\; \leq \; \frac{vol_{\mathsf E}(B^{n-1})}{vol_{\mathsf E}(S^{n-1})} \cdot \mathsf{gd}(M, g) \cdot vol_{g|}(\d M). \qquad  \qquad   \hfill \diamondsuit 
\end{eqnarray}
\end{corollary}

Corollary \ref{diameter} should be compared with the following classical result of Christopher Croke \cite{Cr}. At the first glance, Theorem \ref{th.CROKE} below looks very similar, but it is based on a quantity $\mathsf {d}(M, g|_M)$ that behaves very differently from $\mathsf{gd}(M, g)$, as illustrated in Fig.3. 
\begin{theorem}{\cite{Cr}} \label{th.CROKE}
Let $M$ be a codimension zero compact smooth submanifold of a closed Riemannian $n$-manifold $(N, g)$. Assume that  the diameter $\mathsf {d}(M, g|_M) < \mathsf{inj\, rad}(N, g),$ the injectivity radius of $(N, g)$. Then
 
$$vol_g(M)\; \leq \; \frac{vol_{\mathsf E}(S^{n})}{2\pi\, vol_{\mathsf E}(S^{n-1})} \cdot \mathsf {d}(M, g|_M) \cdot vol_{g|}(\d M). \qquad  \qquad  \hfill \diamondsuit $$  
\end{theorem}

The  conjectures below, which accompany the Definitions \ref{L-volume}-\ref{Swiss}, are quite speculative. 

\begin{definition}\label{L-volume}
Given a compact  Riemannian manifold $(N, g)$, consider the number 
\begin{eqnarray}
\Lambda(N, g) =_{\sf{def}} \sup_{\{(M,\; g|_M)\, \subset \, (N,\; g) \big|\; g|_M \text{ being non-trapping}\}} \, vol_{g|_M}(M)
\end{eqnarray}
and call it the {\sf Lyapunov volume} of $(N, g)$.
Evidently, $\Lambda(N, g) \leq vol_g(N)$. \hfill $\diamondsuit$
\end{definition}

Based on the example of a flat torus and some arguments in \cite{K5}, we formulate 
\begin{conjecture}\label{{conj-L-volume}} For any compact Riemannian manifold $(N, g)$, $\Lambda(N, g) = vol_g(N)$. \hfill $\diamondsuit$
\end{conjecture}

\begin{definition}\label{iso(N, g)} For a given closed Riemannian manifold $(N, g)$, consider the following scale-invariant quantity:
$$
\mathsf {A}(N, g) =_{\sf{def}} \sup_{\{(M,\, g|_M)\, \subset \, (N, g)\}}\, \Big\{\frac{vol_{g|_M}(M)}{vol_{g|_{\d M}}(\d M)\cdot \mathsf{gd}(M, g|_M)}\Big\},
$$
where $(M, g|_M)$ runs over all codimension zero smooth compact submanifolds $M \subset N$ such that $g|_M$ is 
non-trapping. (Note that, for a very small convex $M$, the expression in the brackets $\{\sim\}$ 
approaches an universal positive constant.) 
\smallskip

We also introduce the function of the numerical parameter $V \in [0,\, \Lambda(N, g))$ by 
$$
\mathsf {B}(N, g, V) =_{\sf{def}} \inf_{\{(M,\, g|_M) \subset (N, g)|\; vol_{g|_M}(M) = V\}}\; \big\{\mathsf{gd}(M, g|_M) \cdot vol_{g|_{\d M}}(\d M)\big\},
$$
where $(M, g|_M)$ runs over all codimension zero smooth compact submanifolds $M \subset N$ such that $g|_M$ is 
non-trapping, and the volume of $M$ is a fixed number $V$.\smallskip

The function $\mathsf {B}(N, g, V)$ measures how ``tightly" one can isometrically embed  non-trapping domains $(M, g|_M)$ of volume $V$ into a given $(N, g)$.
\hfill $\diamondsuit$
\end{definition}

So, with Definition \ref{iso(N, g)} in place, formula (\ref{isoM}) leads instantly to the following estimate.

\begin{corollary}  Let $(N, g)$ be a closed Riemannian manifold. Then, for $V \in [0, \Lambda(N, g))$,
$$ 
\mathsf {A}(N, g)\, \leq \, \frac{1}{2\sqrt{\pi}} \cdot \frac{\Gamma(\frac{n}{2})}{\Gamma(\frac{n+1}{2})}, \qquad
\mathsf {B}(N, g, V)\, \geq \, 2\sqrt{\pi} \cdot \frac{\Gamma(\frac{n+1}{2})}{\Gamma(\frac{n}{2})} 
\cdot V. \qquad  \qquad  \hfill \diamondsuit
$$
\end{corollary}

With the same mindset, 
we propose the following definition.

\begin{definition}\label{Swiss}
Consider a compact $n$-dimensional Riemannian manifold $(N, g)$ and a finite collection $\mathcal B := \{B_\a\}_\a$ of disjoint closed smooth balls $B_\a$ such that: 

\begin{itemize}
\item each ball $B_\a$ is geodesicaly strictly convex in $N$,  

\item each connected component of the intersection of any geodesic curve $\g$ in $N$  with the complement $M_\mathcal B =_{\sf{def}}  N \setminus \big(\bigcup_\a \mathsf{int}(B_\a)\big)$ is a closed segment or a singleton (so $g|_{M_\mathcal B}$ is non-trapping). 
\end{itemize}

We call the Riemannian manifold $(M_\mathcal B, g|)$ {\sf the geodesic Swiss cheese model of} $(N, g)$.  \smallskip

We call the number $\mathsf{SC}(N, g) =_{\sf{def}}  \sup_{\{M_\mathcal B\}} vol_g(M_\mathcal B)$ the {\sf Swiss cheese volume} of $(N, g)$.

\hfill $\diamondsuit$
\end{definition}

Evidently, $\mathsf{SC}(N, g) \leq \Lambda(N, g)$; we do not anticipate the two quantities to be equal.

\begin{conjecture} Any compact Riemannian manifold $(N, g)$ admits a geodesic Swiss cheese model. \hfill $\diamondsuit$
\end{conjecture}

For $(N, \mathcal B, g)$ as in Definition \ref{Swiss}, the geodesic vector field $v^g$ on $SM_\mathcal B := \pi^{-1}(M_\mathcal B)$ is traversing and boundary concave. Thus, $v^g$ is boundary generic with respect to the union of the tori $\d(SM_\mathcal B) = S^{n-1} \times \big(\bigcup_\a \d B_\a\big)$ and admits  Lyapunov functions both in $SM_\mathcal B$ and in $S\mathcal B =_{\mathsf{def}} \pi^{-1}(\mathcal B) \approx S^{n-1} \times \mathcal B$. 
By \cite{Si}, \cite{Si1}, \cite{ChM}, and \cite{BFK}, 
the billiard on $M_\mathcal B$ is dispersing and thus ergodic, provided that $g$ is Euclidean.  Therefore, Corollary \ref{cor.FXT} is applicable to such a Swiss cheese billiard. 
\smallskip

Assuming that we know the $g$-induced volume of $\d M$ and the \emph{time record} of the billiard ball hitting the boundary $\d M$ (say, as an infinite sequence of bell rings that broadcast each collision of the billiard ball with the boundary of a curved billiard table $(M, g)$), for an ergodic billiard and a non-trapping $g$, we can ``hear" the volume $vol_g(M)$!

\begin{corollary}\label{cor.FXT} 
For a non-trapping boundary generic $g$ and an ergodic billiard map $B_{v^g}$, knowing 
the limit  $$\mathsf{av}(\ell_{gg}) = \lim_{m \to \infty}\; \frac{1}{m}\, \sum_{k=0}^{m-1} \ell_{gg}\big((B_{v^g})^{\circ k}(z)\big),$$ for almost any $z \in \d_1^+(SM)$, allows to determine the isoperimetric proportion  
$\frac{vol_g(M)}{vol_g(\d M)}$.

The time intervals $\big\{\ell_{gg}\big((B_{v^g})^{\circ k}(z)\big)\big\}_{k \in \Z_+}$ and thus  quantity $\mathsf{av}(\ell_{gg})$ are data, accessible to an observer who is confined to the boundary $\d M$. \hfill $\diamondsuit$
\end{corollary}

\noindent {\bf Question 7.1.} What are other metric quantities of $(M, g)$ that can be recovered from the sequence $\big\{\ell_{gg}\big((B_{v^g})^{\circ k}(z)\big)\big\}_{k \in \Z_+}$ of time intervals (these data may be registered at $\d M$) for ergodic non-trapping billiards? 
\hfill $\diamondsuit$
\smallskip

Combining Theorem \ref{balanced F-variation} with Theorem \ref{ell(gg)} leads to the following claim. 

\begin{corollary}\label{BLA} For a well-balanced Lyapunov function $F: SM \to \R$ and a $v^g$-harmonizing metric $g^\bullet$ on $SM$ such that $\ast_{g^\bullet}(dF) = \om_g^{n-1}$, we get  $$vol_{g^\bullet}(SM) = vol_{gg}(SM) = vol_{\mathsf E}(S^{n-1}) \cdot vol_g(M).$$

Thus, for an ergodic billiard map $B_{v^g}$ and for almost any $z \in \d_1^+(SM)$, the numerical sequence $\big\{\D_F \big((B_{v^g})^{\circ k}(z)\big)\big\}_{k \in \Z_+}$ and $vol_g(\d M)$  allow to reconstruct $vol_g(M)$.
\hfill $\diamondsuit$
\end{corollary}

\begin{theorem}\label{A(g_bullet)} Let $(M, g)$ be a compact connected Riemannian $n$-manifold with boundary, such that the metric $g$ is boundary generic and non-trapping. Let $F: SM \to \R$ be a smooth Lyapunov function for $v^g$, and $var(F)$ its variation on $SM$. We denote by $g^\bullet$ a $v^g$-harmonizing metric on $SM$ such that $\ast_{g^\bullet}(dF) = \om_g^{n-1}$, where $\om_g$ is the restriction of the symplectic form on $TM$ to $SM$. 
 
For each $t \in \R$, consider the minimal hypersurface $F^{-1}(t) \subset SM$ and its $g^\bullet$-induced $(2n-2)$-volume $A_{g^\bullet}(t)$. This construction gives rise to a well-defined measurable function $A_{g^\bullet}: F(SM) \to \R_+$ on the segment $F(SM) \subset \R$. \smallskip

With these ingredients in place, the following statements hold:
\begin{itemize}
\item
\begin{eqnarray}\label{isoF} 
\quad A_{g^\bullet}(t) =  \int_{F^{-1}(t)} \om_g^{n-1} = \Big | \int_{F^{-1}(t)\, \cap \, \d(SM)} \b_g \wedge \om_g^{n-2}\Big |\, \leq \, vol_{\{\om_g^{n-1}\}}\big(\mathcal T(v^g)\big).
\end{eqnarray}

\item The average of the $g^\bullet$-induced volumes $A_{g^\bullet}$ of the $F$-constant lever sections can be calculated via the formula
\begin{eqnarray}\label{A}
\mathsf{av}(A_{g^\bullet}) =_{\mathsf{def}}\;  \frac{\int_{F(SM)}\big(\int_{F^{-1}(t)} \om_g^{n-1}\big)\, dt}{var(F)}  \nonumber \\
= \frac{\int_{SM} dF \wedge \om_g^{n-1}}{var(F)}   = \frac{vol_{g^\bullet}(SM)}{var(F^\d)}.  
\end{eqnarray}
\end{itemize}
\end{theorem}

\begin{proof} For a $v^g$-harmonizing metric $g^\bullet$ that satisfies the hypotheses of the theorem, \hfill \break $dg^\bullet|_{F^{-1}(t)} = \om_g^{n-1}$. Thus $A_{g^\bullet}(t) = \int_{F^{-1}(t)} \om_g^{n-1}$. Since $\om_g^{n-1} = \pm d(\b_g \wedge \om_g^{n-2})$, by Stokes' theorem, 
 $$\int_{F^{-1}(t)\, \cap \, \d(SM)} \b_g \wedge \om_g^{n-2} = \int_{F^{-1}(t)} \om_g^{n-1}.$$ 

Using that $F(v^g) > 0$, we conclude that $F(SM) = F(\d(SM))$. Consider the measurable set $\mathcal X(t)$ of $v^g$-trajectories that have a nonempty intersection with the compact locus $F^{-1}(t)$, where $t \in F(\d(SM))$. Since  $\mathcal X(t) \subset \mathcal T(v^g)$, its $\om_g^{n-1}$-induced measure does not exceed the measure of $\mathcal T(v^g)$. Therefore we get the isoperimetric inequality (\ref{isoF}) for all $t \in F(\d(SM))$ and constant level hypersurfaces $F^{-1}(t)$:  $$\int_{F^{-1}(t)\, \cap \, \d(SM)} \b_g \wedge \om_g^{n-2} = \int_{F^{-1}(t)} \om_g^{n-1} \leq \int_{\d_1^+(SM)} \om_g^{n-1} =_{\mathsf{def}} vol_{\{\om_g^{n-1}\}}\big(\mathcal T(v^g)\big).$$ This validates the claim in the first bullet.\smallskip

Since $F$ has no critical values and the critical values of $F^\d$ have zero measure, we may treat $F: SM \to \R$ as ``almost a fibration", whose fibers are compact $(2n-2)$-manifolds. So the integration over the $F$-fibers is well-defined. Therefore the Fubini formula holds: $\int_{F(SM)}\big(\int_{F^{-1}(t)} \om_g^{n-1}\big)\, dt =  \int_{SM} dF \wedge \om_g^{n-1}$, which proves  (\ref{A}). 
\end{proof}

\begin{corollary}\label{D} 
Under the hypotheses and notations of Theorem \ref{A(g_bullet)}, we get 
\begin{itemize}
\item
 $$vol_{g^\bullet}(SM) \leq  var(F^\d) \cdot \int_{\d_1^+(SM)} \om_g^{n-1} = var(F^\d) \cdot vol_{\{\om_g^{n-1}\}}\big(\mathcal T(v^g)\big).$$
\item The \emph{proportion} 
\begin{eqnarray}\label{C}
\frac{\mathsf{av}(A_{g^\bullet})}{\mathsf{av}(\ell_{g^\bullet})}\; = \; \frac{vol_{\om_g^{n-1}}\big(\mathcal T(v^g)\big)}{var(F^\d)},
\end{eqnarray}
and thus does not depend on the choice of the $(v^g, dF)$-harmonizing metric $g^\bullet$ that satisfies the hypotheses of Theorem \ref{A(g_bullet)}.  
In fact, the proportion in (\ref{C}) depends only on the data that are confined to the boundary $\d(SM)$. \smallskip

\item If, in addition, $F$ is well-balanced, then a $g^\bullet$-independent inequality is valid:  $$ vol_g(M)\; \leq \;  \frac{1}{vol_{\mathsf E}(S^{n-1})} \cdot var(F^\d) \cdot vol_{\{\om_g^{n-1}\}}\big(\mathcal T(v^g)\big).$$
\end{itemize}
\end{corollary}

\begin{proof} The claim in the first bullet of Corollary \ref{D} follows from (\ref{isoF}).
Formula (\ref{C}) is implied by combining (\ref{A}) with (\ref{AA}).
For a well-balanced $F$ and the $(v^g, dF)$-harmonizing $g^\bullet$, we have $\mathsf{av}(\ell_{gg}) = \mathsf{av}(\ell_{g^\bullet})$. Therefore, the inequality in the third bullet follows from the inequality in the first bullet and formula (\ref{ell_gg}).
\end{proof}

\noindent{\bf Remark 8.1} Contemplating about the difference $vol_{\{\om_g^{n-1}\}}\big(\mathcal T(v^g)\big) - \max_{t \in F(SM)} \{A_{g^\bullet}(t)\}$ between the volume of the trajectory space and the maximum of the volumes of the $F$-slices, we may view it as  measuring the complexity of the geodesic flow. More accurately, the difference measures  ``how slanted on average" is the geodesic flow $v^g$ with respect to $\d_1^+(SM)$. \hfill $\diamondsuit$ 

\end{document}